\documentclass[conference,letterpaper]{IEEEtran}

\usepackage[utf8]{inputenc} 
\usepackage[T1]{fontenc}
\usepackage{url}
\usepackage{ifthen}
\usepackage{cite}
\usepackage[cmex10]{amsmath} %

\newcommand{\citefull}{}

\usepackage{amsfonts}
\usepackage{amsmath}
\usepackage{amssymb}
\usepackage{blindtext}
\usepackage{centernot}
\usepackage{array}
\usepackage{graphicx}

\usepackage{verbatim}
\usepackage{float}
\usepackage{upgreek}
\usepackage{mathdots}
\usepackage{multicol}
\usepackage{multirow}

\usepackage{tikz} 
\usepackage{pgfplots}
\pgfplotsset{compat=1.17}
\usepackage{pgfplotstable}
\usepackage{cuted}
\usepgfplotslibrary{fillbetween}
\usetikzlibrary{patterns}

\usepackage{dirtytalk}

\usepackage{bbm}

\usepackage{thmtools}
\usepackage{thm-restate}
\usepackage{mathtools, cuted}
\usepackage{multicol}

\usepackage{url}

\newcommand{\envalias}[2]{\newenvironment{#1}{\begin{#2}}{\end{#2}}}
\envalias{proof}{IEEEproof}

\newcommand{\N}{\mathbb{N}}

\newcommand{\Prob}{\mathbb{P}}

\newcommand{\ra}{\rightarrow}
\newcommand{\R}{\mathbb{R}}

\newtheorem{remark}{Remark}

\newtheorem{theorem}{Theorem}[section]
\newtheorem{lemma}[theorem]{Lemma}
\newtheorem{corollary}[theorem]{Corollary}
\newtheorem{definition}[theorem]{Definition}
\newtheorem{proposition}[theorem]{Proposition}

\newcommand{\runningexample}[1]%
{
\textbf{Running example:}
\textit{#1}
}

\begin{document}
\title{Stabilizability of Vector Systems with Uniform Actuation Unpredictability}

\author{%
  \IEEEauthorblockN{Rahul Arya, Chih-Yuan Chiu, and Gireeja Ranade}
  \IEEEauthorblockA{University of California, Berkeley \\
                    Email: \{rahularya@,chihyuan\_chiu@,ranade@eecs.\}berkeley.edu}
}
\maketitle

\begin{abstract}

Control strategies for vector systems typically depend on the controller's ability to plan out future control actions. However, in the case where model parameters are random and time-varying, this planning might not be possible. 

This paper explores the fundamental limits of a simple  system, inspired by the intermittent Kalman filtering model, where the actuation direction is drawn uniformly from the unit hypersphere. The model allows us to focus on a fundamental tension in the control of underactuated vector systems --- the need to balance the growth of the system in different dimensions.

We characterize the stabilizability of $d$-dimensional systems with symmetric gain matrices by providing tight necessary and sufficient conditions that depend on the eigenvalues of the system. The proof technique is slightly different from the standard dynamic programming approach and relies on the fact that the second moment stability of the system can also be understood by examining any arbitrary weighted two-norm of the state.

\end{abstract}

\section{Introduction}
\label{sec: intro1}

The use of decentralized control for the observation and control of cyberphysical systems, biomedical processes and other engineering applications has inspired much interest in modeling parameter uncertainty and communication bottlenecks in control systems. This paper studies a simple underactuated control system \eqref{Eqn:introsys} 
\begin{align} \label{Eqn:introsys}
    \mathcal{S}: \hspace{5mm} \textbf{X}[n+1] &= A \textbf{X}[n] + \textbf{B}[n] u[n], \\ \nonumber
    \textbf{Y}[n] &= \textbf{X}[n],
\end{align}
where the direction for the control action changes uniformly at random at each time step. Here, $\textbf{X}[n]\in \mathbb{R}^d$ denotes the state and the symmetric matrix $A \in \mathbb{R}^{d\times d}$ is gain of the system $\mathcal{S}$. The actuation direction $\textbf{B}[n]\in \mathbb{R}^d$ is drawn uniformly at random from the unit hypersphere, and $u[n]$ is the scalar control that chooses the magnitude of the step taken in direction $\textbf{B}[n]$ at time $n$.

Since the system model is changing at every time step, the controllers cannot plan for future actions. This model allows us to understand the impact of this inability to accurately plan, while having to stabilize system growth in multiple directions simultaneously

The impact of parameter uncertainty in control and estimation is highlighted by the intermittent Kalman filtering model~\cite{sinopoli2004kalman, Park2011Intermittent} and related works on dropped control packets~\cite{Schenato2007FoundationsofControlandEstimation, Elia2005Remote, imer2006optimal, garone2012lqg, matveev2004}. In particular, these works show how critical the ability to plan is for control systems --- while the critical erasure probability for estimation in the intermittent Kalman filtering setting depends only on the maximal eigenvalue of the system, the critical erasure probability of the dual problem with dropped controls depends on the product of all eigenvalues --- thus control can tolerate far less uncertainty than estimation~\cite{Ramnarayan2014SideInformationinControlandEstimation}. This makes intuitive sense, since a controller must commit to an action at every timestep without knowledge of the future, whereas the observer in the estimation problem can revisit its decisions in the future. 

In particular, in the case of dropped controls, the results of~\cite{Schenato2007FoundationsofControlandEstimation, Elia2005Remote} show that $p_e \cdot (\lambda_1^2 \lambda_2^2) \leq 1$ is the necessary and sufficient condition for second-moment stabilizability, where $p_e$ is the probability of dropped controls and $\lambda_1, \lambda_2 >1$ are the eigenvalues of the system. In contrast for the 2D version of~\eqref{Eqn:introsys}, we show that the system is second-moment stabilizable if and only if $\frac{1}{\lambda_1^2 + \lambda_2^2} \cdot (\lambda_1^2 \lambda_2^2) \leq 1$. In our result, as $\lambda_1^2 + \lambda_2^2$ increases (subject to a constant product $\lambda_1 \lambda_2$), the system is easier to stabilize, whereas in the case of dropped controls only the product of the eigenvalues matters. Similarly, in the $d$-dimensional case, it is not just the total growth (the product of eigenvalues) but also the ``distribution'' of the eigenvalues that matters.

\subsection{Related work}
Parameter uncertainty in systems has been long studied by the control and communication communities, and we build on this work. Some of the most fundamental results were established for the stochastic case through the Uncertainty Threshold Principle~\cite{Athans1977UncertaintyThreshold}, and also in the robust control literature~\cite{zhou1998essentials}, as well as in works such as~\cite{Todorov2003AMinimalIntervention, okano2014arxiv}. These contrast with works such as \cite{brockett2000, tatikonda, anytime, garone2007lqg, qiu2013} that study control in the presence of communication channels, which cause errors, as opposed to parameter randomness from the model itself. Informational and communication bottlenecks in control have been extensively studied, and the books such as~\cite{yuksel2013stochastic, fangtowards} as well as ~\cite{ranade2014active} provide a detailed overview. Older works that considered uncertain parameters in systems but most have limited their understanding to time invariant linear controllers~\cite{willems1976, rajasekaran1971optimum, dekoning1992}. The controller we propose is time-varying but linear, with a dependence on the random direction realized at each time.

Many previous works around parameter uncertainty in control systems highlight the importance of the controller to plan and have worked towards developing an understanding of the ``value of information in control''~\cite{witsenhausen1971separation, rockafellar1976nonanticipativity,dempster1981expected, flaam1985nonanticipativity, davis1989anticipative, back1987shadow}. More recent work has examined the importance of side-information in control problems~\cite{Martins2007FundamentalLimitationsofDisturbanceAttenuation, garone2012lqg, Ramnarayan2014SideInformationinControlandEstimation, Ranade2015ControlCapacity, Ranade2016ControlCapacityForSideInformation, Li2019OnlineOptimalControlwithPredictions}.

Control capacity~\cite{Ranade2015ControlCapacity} considers how parameter uncertainty can explicitly be thought of as an informational bottleneck, and has been explored primarily only in the scalar context and in limited vector systems. Our current work takes a first step towards expanding the notion of control capacity to more general vector systems.  
The previous works of \cite{Hariyoshi2015VectorControlSystems, Hariyoshi2016ControlwithActuationAnticipation} 
studied the control capacity (i.e. the maximum growth rate that a linear system with uncertain actuation can tolerate~\cite{Ranade2015ControlCapacity}) of a system that is defined similar to~\eqref{Eqn:introsys}; however, in addition, after every time step the system is multiplied by a $d\times d$ rotation matrix $\Phi[n+1]$ that uniformly spins the entire system in $\mathbb{R}^d$, i.e. $\textbf{X}[n+1] = \Phi[n+1] \textbf{X}'[n+1]$, where $\textbf{X}'[n+1]$ is a shadow state given as \mbox{$\textbf{X}'[n+1] = A \textbf{X}[n] + \textbf{B}[n] u[n]$}. 
Ordinarily, the matrix gain $A$ of a vector system
affects the orientation of the system state as well as its magnitude. The authors of \cite{Hariyoshi2015VectorControlSystems} circumvent this issue by introducing the random rotation matrix $\Phi[n+1]$ that \say{resets} the orientation of the state $\textbf{X}[n]$ at each time step and maintains the isotropy of the system. The current paper does away with the simplifying $\Phi[n+1]$ that mixes the eigenvalues. The condition for the 2D case in~\cite{Hariyoshi2015VectorControlSystems} states that the system is stabilizable if $\lambda_1^2 + \lambda_2^2 \leq 4$; thus a system with $\lambda_1 = 1.1$ and $\lambda_2 = 2.4$ is not stabilizable under the model in~\cite{Hariyoshi2015VectorControlSystems} but is stable in our current model. On the other hand, we can also identify higher dimensional systems where the system would be stable with the extra spin per~\cite{Hariyoshi2015VectorControlSystems} but are not stable in our current model (e.g. $\lambda_1 = \lambda_2 = 0.5$ and $\lambda_3=\lambda_4 = 1.5$). This is because in some cases the spinning advantages the system by \say{averaging out} large eigenvalues, but also prevents it from fully taking advantage of small eigenvalues, since those also get averaged out by the spin.

\subsection{Main contributions}
Our main results characterize the second-moment stability of a $d$-dimensional system with a symmetric gain matrix.
The system is second-moment stabilizable if \mbox{$r = (m-1)\left(\sum_{i:|\lambda_i|>1} \lambda_i^{-2}\right)^{-1} < 1$}, and only if $r \le 1$, where $m$ is the number of eigenvalues such that $|\lambda_i| > 1$.

Our proof approach relies on an important observation: $\mathbf{X}[n]^T \mathbf{X}[n]$ is uniformly bounded if and only if $\mathbf{X}[n]^T P \mathbf{X}[n]$ is bounded, for any positive definite matrix $P$. That is, if a system is stable, it must be stable in any basis representation. 
When this \say{weight matrix} $P$ is appropriately selected, a greedy control strategy that aims to minimize the weighted norm $\mathbf{X}[n]^T P \mathbf{X}[n]$ at each time step in fact is optimal, in the sense that it stabilizes the system whenever it is possible for any control strategy to do so.

Section \ref{sec: Problem_formulation} formulates the problem, and Section~\ref{sec: results} provides main results, which are specialized to the 2D case in Section~\ref{sec: main_result_2D}. Section \ref{sec: main_result_General} provides proofs of the main lemmas and theorems. Section \ref{sec: conclusion_future_work} presents future directions.

\section{Problem Setup}
\label{sec: Problem_formulation}

Consider the following system setup:
\begin{align} \label{Eqn: System Dynamics}
   \mathcal{S}: \hspace{5mm} \textbf{X}[n+1] &= A \textbf{X}[n] + \textbf{B}[n]u[n], \\ \nonumber 
    \textbf{Y}[n] &= \textbf{X}[n],
\end{align}
where, for each time $n$, $\textbf{X}[n] \in \R^d$ denotes the system state, perfectly observed through $Y[n] \in \R^d$, with the initial state $X[0]$ arbitrarily fixed. Here, $u[n] \in \R$ denotes the scalar control, selected by the controller using $\textbf{B}[0], \cdots, \textbf{B}[n], \textbf{Y}[0], \cdots, \textbf{Y}[n], u[0], \cdots, u[n-1]$, that couples the system. $A$ denotes the fixed, symmetric, non-singular gain matrix, known to the controller, that is applied to the state at each time, while $\textbf{B}[n]$ denotes the random actuation direction drawn i.i.d. and uniformly from $\mathcal{S}_d$, the $d$-dimensional unit hypersphere:
\begin{equation} \label{Eqn: d-dimensional unit hypersphere}
    \mathcal{S}_d := \{(x_1, \cdots, x_d)| x_1^2 + \cdots x_d^2 = 1 \}.
\end{equation}
More formally, $\textbf{B}[n]$ may be component-wise generated i.i.d. from the standard normal distribution, then normalized. Since the distribution of $\textbf{B}[n]$ is rotation-invariant, and the symmetric matrix $A$ is orthogonally diagonalizable, we assume without loss of generality below that $A$ is diagonal.

For each $n \in \N$, we use $\textbf{Y}_0^n$, $\textbf{B}_0^n$, $U_0^n$ to denote the observations, actuation vectors, and controls from times $0$ to $n$, respectively, i.e. \mbox{$\textbf{Y}_0^n := (\textbf{Y}[0], \cdots, \textbf{Y}[n])$} etc.  %
Let:
\begin{align}
    \mathcal{G}_n := \{ g[n](\cdot): (\textbf{B}_0^n, \textbf{Y}_0^n, U_0^{n-1}) \rightarrow u[n] \}
\end{align}
denote the set of all permissible control strategies at time $n$. Define by $g_0^n(\cdot)$ a permissible control strategy from time 0 to time $n$ (inclusive), and denote by $\mathcal{G}_0^n := (\mathcal{G}_0, \cdots, \mathcal{G}_n)$ the set of all such control strategies from time 0 to time $n$.

At each time $n$, the following sequence of events occurs:
\begin{enumerate}
    \item $\textbf{B}[n]$ is drawn uniformly from $\mathcal{S}_d$, independent of all random vectors up to time $n$.
    
    \item The controller observes $\textbf{B}[n]$ and selects $u[n]$ based on $\textbf{Y}_0^{n}, \textbf{B}_0^n, U_0^{n-1}$.
    
    \item The system propagates to the next state, i.e. $\textbf{X}[n+1]$ and $\textbf{Y}[n+1]$ are generated according to the dynamics~\eqref{Eqn: System Dynamics}.  
\end{enumerate}

The controller aims to stabilize the system $\mathcal{S}$ in a second-moment sense, as defined below.

\begin{definition}[\textbf{Second-Moment Stabilizable}] \label{Def: Second-Moment Stabilizable}
The system $\mathcal{S}$ is said to be \textbf{second-moment stabilizable} if there exists a sequence of controls $U_0^\infty$, such that:
\begin{align}
    \limsup_{n \ra \infty} \mathbb{E}[\textbf{\emph{X}}[n]^T  \textbf{\emph{X}}[n]] < \infty.
\end{align}
\end{definition}

We will also use the notion of equivalent norms, induced by a symmetric positive definite matrix $P$, to discuss second-moment stabilizability.

\begin{proposition} \label{Prop: Change of Coordinates, 2nd-Moment Stabilizability}
Let $\textbf{\emph{Z}}[n]$ be a $\R^d$-valued random variable for each $n \in \N$, and let $P \in \R^{d \times d}$ be any symmetric positive definite matrix.
Then, for each $n \in \N$:
\begin{align} \label{Eqn: Prop, Change of Coordinates, Norm Equivalence}
    \Vert P^{-1} \Vert_2^{-1} \cdot \Vert \textbf{\emph{Z}}[n] \Vert_2^2 \leq \textbf{\emph{Z}}[n]^T P \textbf{\emph{Z}}[n] \leq \Vert P \Vert_2 \cdot \Vert \textbf{\emph{Z}}[n] \Vert_2^2.
\end{align}
In particular, $\limsup\limits_{n \ra \infty} \mathbb{E}\big[\textbf{\emph{Z}}[n]^T \textbf{\emph{Z}}[n] \big] < \infty$ if and only if $\limsup\limits_{n \ra \infty} \mathbb{E}\big[\textbf{\emph{Z}}[n]^T P \textbf{\emph{Z}}[n] \big] < \infty$.

\end{proposition}

\begin{proof}
The linear algebra proof details are in \mbox{Appendix~\ref{AppSubsec: Prop: Change of Coordinates, 2nd-Moment Stabilizability}}\citefull.
\end{proof}

\emph{Notation:} In this paper, lower-case letters (e.g. $r, m, w$) denote scalars, while upper-case letters (e.g. $M, W$) denote vectors or matrices. Subscripts indicate matrix or vector elements, e.g. $M_{ij}$ indicates the $(i, j)$-th element of the matrix $M$. Boldface letters (e.g. \textbf{B}) indicate random variables. %

\section{Main Results}
\label{sec: results}

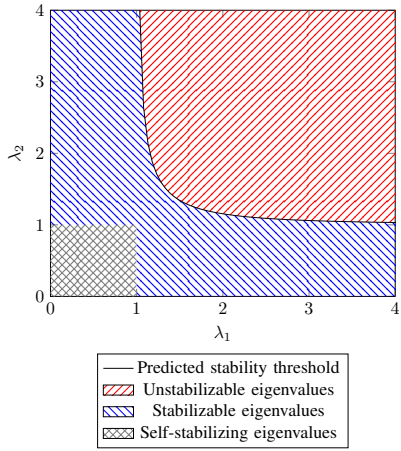
\begin{figure}[t]
\centering
\resizebox{0.3\textwidth}{!}{\begin{tikzpicture}
    \begin{axis}[
        xlabel=$\lambda_1$, ylabel=$\lambda_2$,
        xmin=0, xmax=4,
        ymin=0, ymax=4,
        legend style={at={(0.5,-0.2)},anchor=north}
     ]
     \addlegendentry{Predicted stability threshold}
     \addlegendentry{Unstabilizable eigenvalues}
     \addlegendentry{Stabilizable eigenvalues}
     \addlegendentry{Self-stabilizing eigenvalues}
    \path[name path=high] (axis cs:0,4) -- (axis cs:4,4);
    \addplot[name path=thresh, samples=200]{sqrt(x^2 / (x^2 - 1))};
    \path[name path=box] (axis cs:0,1) -| (axis cs:1,0) -- (axis cs:4,0);
    \path[name path=low] (axis cs:1,0) -| (axis cs:0,1);
    \addplot[pattern=north east lines, pattern color=red] fill between[of=thresh and high];
    \addplot[pattern=north west lines, pattern color=blue] fill between[of=box and thresh];
    \addplot[pattern=crosshatch, pattern color=gray] fill between[of=low and box];
 \end{axis}
\end{tikzpicture}}
\caption{Threshold for stabilizability for a 2D system with $A = \text{diag}(\lambda_1, \lambda_2)$.}
\label{fig: 2d-stability}
\end{figure}

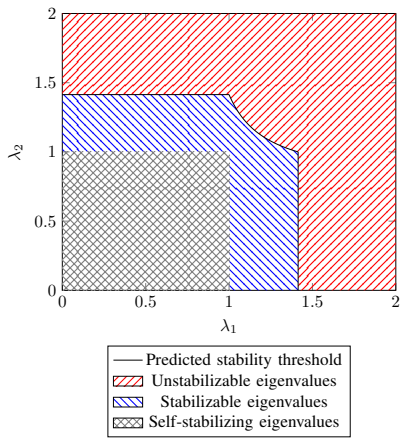
\begin{figure}[t]
\centering
\resizebox{0.3\textwidth}{!}{\begin{tikzpicture}
    \begin{axis}[
        xlabel=$\lambda_1$, ylabel=$\lambda_2$,
        xmin=0, xmax=2,
        ymin=0, ymax=2,
        legend style={at={(0.5,-0.2)},anchor=north}
     ]
     \addlegendentry{Predicted stability threshold}
     \addlegendentry{Unstabilizable eigenvalues}
     \addlegendentry{Stabilizable eigenvalues}
     \addlegendentry{Self-stabilizing eigenvalues}
    \path[name path=high] (axis cs:0, 1.414) -- (axis cs:0,2) -| (axis cs:2,0) -- (axis cs:1.414,0);
    \addplot +[mark=none, black,forget plot] coordinates {(0, 1.414) (1, 1.414)};
    \addplot +[mark=none, black,forget plot] coordinates {(sqrt(2), 0) (sqrt(2), 1)};
    \addplot[name path=thresh, samples=200, domain=1:sqrt(2)]{sqrt(2 * x^2 / (3*x^2 - 2))};
    \path[name path=upperbox] (axis cs:1.414,0) -- (axis cs:1,0) |- (axis cs:0,1) -- (axis cs:0,1.414);
    \path[name path=box] (axis cs:0,1) -| (axis cs:1,0);
    \path[name path=low] (axis cs:1,0) -| (axis cs:0,1);
    \addplot[pattern=north east lines, pattern color=red] fill between[of=thresh and high];
    \addplot[pattern=north west lines, pattern color=blue] fill between[of=upperbox and thresh];
    \addplot[pattern=crosshatch, pattern color=gray] fill between[of=low and box];
    \end{axis}
\end{tikzpicture}}
\caption{Threshold for stabilizability for a 4D system with $A = \text{diag}(\lambda_1, \lambda_1, \lambda_2, \lambda_2)$. Notice that if $\lambda_1 > \sqrt{2}$ or $\lambda_2 > \sqrt{2}$, the system must be unstabilizable.}
\label{fig: 3d-stability}
\end{figure}

Define a diagonal gain matrix $A$ as below:
\begin{align}
    \label{Eqn: A matrix, in Main Results}
    A = \text{diag}\{\lambda_1, \cdots, \lambda_d\}.
\end{align}

\begin{restatable}{theorem}{mainthm}
\label{Thm: Main, d-dimensional}
Suppose $|\lambda_i| \ne 1$ for all $1 \le i \le d$. Let $m$ of these eigenvalues $|\lambda_i| > 1$. Set:
\begin{align}
    \label{Eqn: General Stability Result, Main Results Section}
    r := \frac{m-1}{\sum_{i:|\lambda_i| > 1} \lambda_{i}^{-2}}.
\end{align}
The system $\mathcal{S}$ is second-moment stabilizable if $r < 1$.
Conversely, if $\mathcal{S}$ is second-moment stabilizable, then $r \leq 1$.
\end{restatable}

This threshold for the 2D case is shown in Fig. \ref{fig: 2d-stability}, and the threshold for a 4D system with $A = \text{diag}(\lambda_1, \lambda_1, \lambda_2, \lambda_2)$ is shown in Fig. \ref{fig: 3d-stability}. The plotted points indicate the stabilizability of a numerical simulation of $\mathcal{S}$ using the optimal (in the weighted-norm sense) control strategy we derive.

Our proof follows the following structure. Using dynamic programming, we first show that the optimal control input $u[n]$ at the $n$th time step is to minimize $\|W[n+1]^{1/2}\mathbf{X}[n+1]\|_2$, for some positive definite diagonal matrix $W[n+1]$. We then express $W[n+1]$ as a deterministic function of $W[n]$ using a Riccati-type recursion, independent of the realizations of the $\mathbf{B}[n]$. Under certain conditions~\eqref{Eqn: Lemma, Main, Constraint on lambdas} on the $\{ \lambda_i \}$, we show (without explicitly constructing it) that there exists some \mbox{$P = W[0]$} such that $W[n] = r^n W[0]$ for the scalar $r$ defined in \eqref{Eqn: General Stability Result, Main Results Section}. If $r < 1$, then $\mathbb{E}[\mathbf{X}[n]^TW[n]\mathbf{X}[n]] = r^n \mathbb{E}[\mathbf{X}[0]^TW[0]\mathbf{X][0]}$, which tends to zero as $n \to \infty$, which gives stability.

If these conditions on the eigenvalues are not satisfied, but all the $|\lambda_i| > 1$, we show both that $r > 1$ and that there exists an embedded subsystem $\overline{\mathcal{S}}$ of $\mathcal{S}$ that is not second-moment stabilizable, so $\mathcal{S}$ cannot be stabilizable. If some $|\lambda_i| < 1$, we show that $\mathcal{S}$ is stabilizable only if the embedded subsystem with state matrix $\text{diag}(\{\lambda_i : |\lambda_i| > 1\})$ is stabilizable. Putting these cases together yields the final result. 

It is possible to extend Theorem \ref{Thm: Main, d-dimensional} to apply to the case when some $|\lambda_i| = 1$. We define $r$ analogously to~\eqref{Eqn: General Stability Result, Main Results Section} considering all eigenvalues that are $\geq 1$. If $r<1$ and some $|\lambda_i| = 1$ we may perturb $|\lambda_i|$ to become slightly larger --- a sufficiently small perturbation will keep $r < 1$. Increasing the magnitude of the $\lambda_i$ cannot make an unstabilizable system stabilizable, so our original system must also have been stabilizable. The other direction works similarly.

We hope to extend these results to the general matrices $A$, however, understanding the case where $A$ is symmetric and thus can be diagonalized by an orthonormal basis is a natural first step. In our setup, the control maintains the interactions across dimensions and thus the true vector nature of the problem, but the orthonormal basis representation makes the dynamic program tractable.

\section{An Illustrative 2D Example}
\label{sec: main_result_2D}

We first illustrate our result for the special case of $d=2$. This gives a flavor of the overall proof --- illustrating how an appropriate choice of $P$ lets us determine the stabilizability of the system --- while avoiding the complexities that arise when $d > 2$. In the general $d$-dimensional case, we cannot directly find a choice of $P$ --- instead, roughly speaking, we non-constructively show the existence of an appropriate $P$ under certain conditions, and show that the system is not stabilizable when these conditions are not satisfied (see Section~\ref{sec: main_result_General}).%

\begin{theorem} \label{Thm: Main, 2-dimensional}
Suppose the system $\mathcal{S}$ is given with its gain matrix $A$ defined by
$A := \text{diag}(\lambda_1, \lambda_2)$,
where $\lambda_1, \lambda_2 \ne 0$. Then the system $\mathcal{S}$ is second-moment stabilizable if and only if $r \leq 1$, where $r > 0$ is given by $r = \left(\tfrac{1}{\lambda_1^2} + \tfrac{1}{\lambda_2^2}\right)^{-1}$.
\end{theorem}
\begin{remark}
Notice here we include all eigenvalues in the definition of $r$ instead of just those with magnitudes larger than 1 as in \eqref{Eqn: General Stability Result, Main Results Section} from Thm.~\ref{Thm: Main, d-dimensional}. 
Although the values of $r$ in Thms.~\ref{Thm: Main, d-dimensional} and \ref{Thm: Main, 2-dimensional} are slightly different, the results are consistent. Note that $r$ is not necessarily the optimal decay rate of the system. 
\end{remark}

We begin with some auxiliary results that help us evaluate $\mathbb{E}[\mathbf{X}[N]^TP\mathbf{X}[N]]$ in terms of $\mathbf{X}[n]$ at an earlier timestep $n < N$, by setting up a Ricatti-like recursion. These results are useful in the 2D case as well as the more general proof.

\begin{restatable}{proposition}{ricattiproperties} \label{Prop: Riccati Recursion, Properties}
Fix any $N \in \N$ and any diagonal, symmetric positive definite matrix $P$. Let the sequence of matrices $W[n], \textbf{\emph{M}}[n]$ be defined as below:
\begin{align} \label{Eqn: Thm, Riccati Recursion for W, at time N}
    W[N] &= P, \\ \label{Eqn: Thm, Riccati Recursion for W, Iterative Recursion}
    W[n] &= A^T \big( W[n+1] - \mathbb{E}[\textbf{\emph{M}}[n]] W[n+1] \big) A \\
    \textbf{\emph{M}}[n] &= \frac{W[n+1] \textbf{\emph{B}}[n]\textbf{\emph{B}}[n]^T}{\textbf{\emph{B}}[n]^T W[n+1] \textbf{\emph{B}}[n]} \label{Eqn: Thm, Riccati Recursion for W, M[n] Definition}.
\end{align}
Then each $W[n]$ is diagonal and symmetric positive definite.
\end{restatable}

\begin{proof}
The proof follows by backwards induction from $n = N$, using symmetry and linear algebraic arguments. Details are deferred to Appendix \ref{AppSubsec: Prop: Riccati Recursion, Properties}\citefull.
\end{proof}

\begin{lemma} \label{Lemma: Riccati Recursion}
Consider the system $\mathcal{S}$ starting at the fixed, known initial state $X[0]$ at time $0$. Fix any diagonal, symmetric positive definite matrix $P \in \R^{d \times d}$. Then, for each $N \in \N$, $n = 0, 1, \cdots, N$:
\begin{equation} \label{Eqn: Thm, Riccati Recursion for W, Main Statement}
    \min_{g_0^N \in \mathcal{G}_0^N} \mathbb{E}\left[ \textbf{\emph{X}}[N]^T P \textbf{\emph{X}}[N] \right] = \min_{g_0^n \in \mathcal{G}_0^n} \mathbb{E}\left[ \textbf{\emph{X}}[n]^T W[n] \textbf{\emph{X}}[n] \right],
\end{equation}
where \mbox{$\{W[n] \mid n = 0, 1, \cdots, N\}$} is given by the recursive formulas \eqref{Eqn: Thm, Riccati Recursion for W, at time N} and \eqref{Eqn: Thm, Riccati Recursion for W, Iterative Recursion}, and $W[N] = P$. In particular:
\begin{align} \nonumber
    \min_{g_0^N \in \mathcal{G}_0^N} \mathbb{E}\left[ \textbf{\emph{X}}[N]^T P \textbf{\emph{X}}[N] \right] =  \textbf{\emph{X}}[0]^T W[0] \textbf{\emph{X}}[0].
\end{align}
\end{lemma}

\begin{proof}
\eqref{Eqn: Thm, Riccati Recursion for W, Main Statement} basically states that the minimization over $g_0^N := (g_0, \cdots, g_n, g_{n+1}, \cdots g_N) \in \mathcal{G}_0^N$ can be rewritten into two parts---the minimization over $g_0^n := (g_0, \cdots, g_n)$ and the minimization over $g_{n+1}^N := (g_{n+1}, \cdots, g_N)$---and that the latter can be solved in closed form using dynamic programming, by exploiting the system dynamics \eqref{Eqn: System Dynamics}. The full proof is deferred to Appendix \ref{AppSubsec: Thm: Riccati Recursion}\citefull.
\end{proof}

\begin{lemma} \label{Lemma: Geometric Series with r}
Suppose there exists some $r > 0$ and symmetric positive definite matrix $P$ such that $W[n] = r^{N-n} \cdot P$ for each $n = 0, 1, \cdots, N-1$. Then the system $\mathcal{S}$ is second-moment stabilizable if $r < 1$. Furthermore, if $\mathcal{S}$ is second-moment stabilizable, then $r \le 1$.

\end{lemma}

\begin{proof}
For each $N \in \N$, \eqref{Eqn: Thm, Riccati Recursion for W, Main Statement} from Lemma \ref{Lemma: Riccati Recursion} implies: 
\small{
\[
\min_{g_0^N \in \mathcal{G}_0^N} \mathbb{E}\left[ \textbf{X}[N]^T P \textbf{X}[N] \right] = X[0]^T W[0] X[0] = r^N X[0]^T P X[0].
\]} \normalsize
Since $X[0]$ and $P$ are fixed, the sequence $\Big(\min_{g_0^N \in \mathcal{G}_0^N} \mathbb{E}\left[ \textbf{X}[N]^T P \textbf{X}[N] \right], N = 0, 1, \cdots \Big)$ is bounded if and only if $r \leq 1$ and approaches 0 if and only if $r < 1$. The lemma now follows from Proposition \ref{Prop: Change of Coordinates, 2nd-Moment Stabilizability}.
\end{proof} 

We now proceed to the main proof of Thm.~\ref{Thm: Main, 2-dimensional}. %

\begin{proof}
By Proposition \ref{Prop: Change of Coordinates, 2nd-Moment Stabilizability}, the system is second-moment stabilizable if and only if there exists some positive definite matrix $P$ such that $\lim_{n \ra \infty} \mathbb{E}\big[\textbf{X}[n]^T P \textbf{X}[n] \big] < \infty$.

To this end, set $P = \text{diag}(\lambda_1^4, \lambda_2^4)$ and define $W[N], \cdots, W[0]$ recursively via \eqref{Eqn: Thm, Riccati Recursion for W, Iterative Recursion}, with $W[N] = P$. Let the elements of the actuation vector $\textbf{B}[n]$ be given by \mbox{$\textbf{B}[n] = (\textbf{b}_1, \textbf{b}_2)$}. By Proposition \ref{Prop: Riccati Recursion, Properties}, each $W[n]$ becomes a diagonal, symmetric positive definite matrix
whose diagonal entries $w_{n, i}$ for $i = 1,2$ are given recursively by:
\small

\begin{align} \nonumber
    w_{N, 1} &= \lambda_1^4, \quad w_{N, 2} = \lambda_2^4, \\ 
    \label{Eqn: Thm, Main, 2D, w(n, 1) recursion}
    w_{n, i} &= \lambda_1^2 w_{n+1, i} \left(1 - \mathbb{E}\left[ \frac{\textbf{b}_i[n]^2 \cdot w_{n+1, i} }{\textbf{b}_1[n]^2 \cdot w_{n+1, 1} + \textbf{b}_2[n]^2 \cdot w_{n+1, 2}} \right] \right),
\end{align}

\normalsize

Below, we show via backwards induction that
the sequence of matrices $\{W[n]\}_{n=0}^N$ forms a geometric series with $W[n] = r^{N-n} \cdot P$,
where $r = \left( \frac{1}{\lambda_1^2} + \frac{1}{\lambda_2^2} \right)^{-1}$.
    
To establish this, we must show that
for each $n \in \N$:
\begin{align} \label{Eqn: Thm, Main, 2D, W ratios}
    \frac{w_{n, 1}}{w_{n+1, 2}} = \frac{w_{n, 2}}{w_{n+1, 2}} = r.
\end{align}
Assume that \eqref{Eqn: Thm, Main, 2D, W ratios} holds at time $n+1$, for some $n = 0,1,\cdots,N-1$. Now, define $\alpha := \lambda_2^4/\lambda_1^4$, and note that we can write the expectation in \eqref{Eqn: Thm, Main, 2D, w(n, 1) recursion} as:
\begin{align*}
    &\mathbb{E}\left[ \frac{w_{n+1, 1} \cdot \textbf{b}_1[n]^2}{w_{n+1, 1} \cdot \textbf{b}_1[n]^2 + w_{n+1, 2} \cdot \textbf{b}_2[n]^2} \right] \\
    = \hspace{0.5mm} &\mathbb{E}\left[ \frac{\textbf{b}_1[n]^2}{\textbf{b}_1[n]^2 + \alpha \cdot \textbf{b}_2[n]^2} \right] = \frac{1}{2\pi} \int_0^{2\pi} \frac{\cos^2 \theta}{\cos^2 \theta + \alpha \sin^2\theta} \hspace{0.5mm} d\theta \\
    = \hspace{0.5mm} &\frac{1}{\sqrt{\alpha} + 1} = \frac{\lambda_1^2}{\lambda_1^2 + \lambda_2^2},
\end{align*}
where we have parameterized 
the unit circle in $\R^2$ by using the uniform distribution of the polar angle $\theta$ on $[0, 2\pi)$.
Hence,%
\begin{align*}
    w_{n, 1} 
    &= \lambda_1^2 w_{n+1, 1} \cdot \left(1 - \frac{\lambda_1^2}{\lambda_1^2 + \lambda_2^2} \right) = \frac{\lambda_1^2 \lambda_2^2}{\lambda_1^2 + \lambda_2^2} \cdot w_{n+1, 1}, \\
    w_{n, 2} 
    &= \lambda_2^2 w_{n+1, 2} \cdot \frac{\lambda_1^2}{\lambda_1^2 + \lambda_2^2} = \frac{\lambda_1^2 \lambda_2^2}{\lambda_1^2 + \lambda_2^2} \cdot w_{n+1, 2}.
\end{align*}
Thus, we have:
\begin{align*}
    \frac{w_{n, 1}}{w_{n+1, 1}} = \frac{w_{n, 2}}{w_{n+1, 2}} &=  \frac{\lambda_1^2 \lambda_2^2}{\lambda_1^2 + \lambda_2^2} = \left( \frac{1}{\lambda_1^2} + \frac{1}{\lambda_2^2} \right)^{-1},
\end{align*}
as claimed. Thus, $W[n] = r^{N-n} W[N]$ for each $0 \le n \le N$.
By Lemma \ref{Lemma: Geometric Series with r}, this completes the proof. 
\end{proof}

\section{Proof of Thm.~\ref{Thm: Main, d-dimensional}}
\label{sec: main_result_General}
For the general result, we will consider two cases. In Case 1, all the $|\lambda_i| > 1$. In Case 2, at least one of the eigenvalues $|\lambda_i| < 1$. We focus on $|\lambda_i| \ne 1~\forall i$ here though our results extend for the $|\lambda_i| = 1$ case as mentioned earlier.

Case 1 is itself split into two sub-cases (1a and 1b), depending on whether or not the eigenvalues of our system satisfy a particular technical condition in~\eqref{Eqn: Lemma, Main, Constraint on lambdas}. In Case 1a, the proof proceeds similarly to in the two-dimensional case, except that we only show the existence of an appropriate $P$, such that  $W[n] = r^{N-n} \cdot P$, without constructing it directly. In Case 1b, we  show that a subsystem $\overline{\mathcal{S}}$ of our system exists that is not second-moment stabilizable. 
To handle Case 2, we show how the control strategy from Case 1 can be modified by occasionally dropping controls to stabilize the system.

Detailed proofs are deferred to Appendices \ref{AppSubsec: Lemma: Analysis, Existence of w(N, 1), w(N, 2)} through \ref{AppSubsec: Case 2}\citefull, and we provide sketches here.

\subsection{Case 1a}
\label{subsec: auxiliary lemmas}

In this case we assume all eigenvalues of $A$ satisfy the condition~\eqref{Eqn: Lemma, Main, Constraint on lambdas}. First we investigate conditions such that $W[n] = r^{N-n} P$ for some positive definite $P$. Then we show that when these conditions hold, stability is possible if $r < 1$, but not if $r > 1$. This is formalized in Lemma \ref{Lemma: Main}.%

\begin{restatable}{lemma}{lemmamain} \label{Lemma: Main}
Consider the system $\mathcal{S}$ with fixed, known initial state $X[0]$. If the eigenvalues of $A$ satisfy the condition
\begin{align} \label{Eqn: Lemma, Main, Constraint on lambdas}
    v_i^* := 1 - \frac{(d-1)\lambda_i^{-2}}{\sum_{j=1}^d \lambda_j^{-2}} > 0
\end{align}
for all $1 \le i \le d$, then we can show that there exists some $r \in \R$, $r > 0$ and some diagonal, symmetric positive definite matrix $P \in \R^{d \times d}$,
    $P := \text{diag}\{ p_1, p_2, \cdots, p_d \}$
such that the matrices $W[N], \cdots, W[0]$ generated by the recursion \eqref{Eqn: Thm, Riccati Recursion for W, at time N} from $W[N] = P$ satisfy: $W[n] = r^{N-n} \cdot P$ for each $n = 0, 1, \cdots, N$.
Moreover, $r$ is given by
    \mbox{$r = (d-1)/\sum_{i=1}^d \lambda_i^{-2}.$}
\end{restatable}
Note that this $r$ matches \eqref{Eqn: General Stability Result, Main Results Section} in our Case 1a, when all the $|\lambda_i| > 1$.
The proof of Lemma \ref{Lemma: Main} depends on the technical Lemma~\ref{Lemma: Analysis, Existence of w(N, 1), w(N, 2)}. The proof strategy uses Lemma~\ref{Lemma: Analysis, Existence of w(N, 1), w(N, 2)} to show the existence of an appropriate initial weighting matrix $P$, given that \eqref{Eqn: Lemma, Main, Constraint on lambdas} holds, such that $W[n] = r^{N-n}P$. From this and Lemma \ref{Lemma: Geometric Series with r}, Theorem \ref{Thm: Main, d-dimensional} for Case 1a immediately follows.

\begin{lemma} \label{Lemma: Analysis, Existence of w(N, 1), w(N, 2)}
Let $\textbf{\emph{B}}[n] = (\textbf{\emph{b}}_1[n], \cdots, \textbf{\emph{b}}_d[n])$ be a vector drawn from $\mathcal{S}_d$, the $d$-dimensional hypersphere defined in \eqref{Eqn: d-dimensional unit hypersphere}. Let $v_i$ be positive reals such that $\sum_{i=1}^d v_i = 1$. For all such $v_i$, there exist $p_i > 0$ such that for $1 \le i \le d$:
\begin{align} \label{Eqn: Lemma, Analysis, Existence of w, Eqn. 1}
    \mathbb{E}\left[ \frac{\textbf{\emph{b}}_i[n]^2 \cdot p_i}{\sum_{k=1}^d \textbf{\emph{b}}_k[n]^2 \cdot p_k} \right] = v_i.
\end{align}

\end{lemma}

\begin{proof} (Sketch)
The proof inductively establishes that given the first $\{p_1, \cdots, p_i\}$, we can select the subsequent $\{ p_{i+1}, \cdots p_d \}$ such that the latter $d - i$ equalities hold.
\end{proof}

\vspace{-1mm}
\subsection{Case 1b}
Next, we will consider Case 1b, when the conditions in \eqref{Eqn: Lemma, Main, Constraint on lambdas} are not satisfied. For this we use
Lemma \ref{Lemma: Embedded System} which shows that if $\mathcal{S}$ contains a lower-dimensional system that is unstable, $\mathcal{S}$ must also be unstable.

\begin{lemma} \label{Lemma: Embedded System}
Let $A' = \text{diag}(\lambda_2, \cdots, \lambda_d)$. Define the $(d-1)$-dimensional system $\overline{\mathcal{S}}$ to be similar to \eqref{Eqn: System Dynamics} but with state matrix $A'$. If $\overline{\mathcal{S}}$ is not second-moment stabilizable, then $\mathcal{S}$ is also not second-moment stabilizable. 

\end{lemma}

Now, we will prove Theorem \ref{Thm: Main, d-dimensional} in Case 1b by induction on the dimension $d$. Assume that it holds for all systems of dimension $d - 1$, and consider a $d$-dimensional system. If the $d$-dimensional system falls into Case 1a, we can use the known value of $r$ from Lemma \ref{Lemma: Main}, along with Lemma \ref{Lemma: Geometric Series with r}, to establish conditions on the system's stabilizability. Since \eqref{Eqn: Lemma, Main, Constraint on lambdas} is satisfied when $d = 1$, this also establishes the base case for the induction.

Otherwise, we know that at least one $v_{i}^* \le 0$ --- WLOG, let one such term be $v_{1}^*$. 
From \eqref{Eqn: Lemma, Main, Constraint on lambdas}, we have that $\sum_{i=1}^d v_{i}^* = 1$. Therefore,
\[
    1 \le \sum_{i=2}^d v_{i}^* = \sum_{i=2}^d \left(1 - \frac{(d-1)\lambda_i^{-2}}{\sum_{i=1}^d \lambda_i^{-2}}\right) = (d-1) - r \sum_{i=2}^d \lambda_i^{-2}.
\]
Rearranging, we find that
\begin{align} \label{Eqn: Inductive Inequality Case 2}
    \frac{d-2}{\sum_{i=2}^d \lambda_i^{-2}} \ge r.
\end{align}
Furthermore, since
\[
   0 \ge  v_{1}^* = 1 - \frac{(d-1)\lambda_1^{-2}}{\sum_{i=1}^d \lambda_i^{-2}} = 1 - r\lambda_1^{-2},
\]
we have that $r \ge \lambda_1^2 > 1$, since we assumed that \mbox{$|\lambda_1| > 1$}. From \eqref{Eqn: Inductive Inequality Case 2}, we see that $(d-2)/\sum_{i=2}^d \lambda_i^{-2} > 1$, so we may invoke our inductive hypothesis on the $(d-1)$-dimensional system $\overline{\mathcal{S}}$ with state matrix $A' = \text{diag}(\lambda_2, \cdots, \lambda_d)$ to see that $\overline{\mathcal{S}}$ is not second-moment stabilizable. Then, by Lemma \ref{Lemma: Embedded System}, the system $\mathcal{S}$ is not stabilizable as well. Since $r > 1$, this is the desired conclusion, since we have shown %
that $r$ is not strictly less than $1$, and our system is not stabilizable. %

\subsection{Case 2}
Now, we can consider Case 2, when some of our eigenvalues may be less than $1$ in magnitude.
\begin{proof} (Sketch)
Let $\mathcal{\overline{S}}$ be an $m$-dimensional system embedded in $\mathcal{S}$ whose state matrix includes just those eigenvalues of $A$ greater than $1$ in magnitude. Note that $r$ is the same for $\mathcal{\overline{S}}$ and $\mathcal{S}$. If $\mathcal{S}$ is stabilizable, then by Lemma \ref{Lemma: Embedded System}, so is $\overline{\mathcal{S}}$. Then by following the argument in Case 1, we have $r \le 1$.

Now assume $r<1$ to show the other direciton of the proof. This implies that the subsystem $\mathcal{\overline{S}}$ is stabilizable. We will use a scaled version of the scalar controls that stabilize  $\mathcal{\overline{S}}$ to stabilize $\mathcal{S}$. Here effectively we are ``ignoring'' the additional $(d-m)$ dimensions orthogonal to the embedded system that has eigenvalues $< 1$ in magnitude. Any components of $\textbf{B}[n]$ (and $\textbf{X}[n]$) in those $(d-m)$ directions (with eigenvalues $< 1$) should naturally decay, as long as the components in those directions do not accumulate too fast.

To ensure this, at each time step, if $\mathbf{B}[n]$ is ``almost'' orthogonal to $\mathcal{\overline{S}}$ (meaning that $u[n]$ would be large and $\textbf{B}[n] u[n]$ will be large in the other $(d-m)$ dimensions), we will instead apply a zero control input. 

By slightly modifying the recursion in \eqref{Eqn: Thm, Riccati Recursion for W, at time N} and \eqref{Eqn: Thm, Riccati Recursion for W, Iterative Recursion} to handle the case of dropped controls, and correspondingly modifying and re-proving Proposition \ref{Prop: Riccati Recursion, Properties}, Lemma \ref{Lemma: Riccati Recursion} and Lemma \ref{Lemma: Main}, we can show that if $r < 1$, $\mathcal{\overline{S}}$ can be stabilized even if controls are dropped with some arbitrarily small, but nonzero probability. The resultant bounded input in the orthogonal $d-m$ dimensions means that $\mathbf{X}[n]$ is bounded in the limit in every dimension, so $\mathcal{S}$ is also stabilizable if $r < 1$, which proves the result.\end{proof}

\section{Conclusions and Future Work}
\label{sec: conclusion_future_work}
This work takes a first step to understand how the inability to plan due to parameter uncertainty affects vector systems that must trade-off between growth in different dimensions. We hope to extend this to understand non-symmetric gain matrices as well side-information and lookahead in systems to build our understanding of the value-of-information in control settings.

\newpage
\enlargethispage{-7.5cm}

\bibliographystyle{IEEEtran}
\bibliography{references}

\clearpage
\section{Appendices}
\label{sec: Appendices}

\subsection{Proof of Proposition \ref{Prop: Change of Coordinates, 2nd-Moment Stabilizability}} \label{AppSubsec: Prop: Change of Coordinates, 2nd-Moment Stabilizability}
Fix $V \in \R^d$ and $P \in \R^d$ arbitrarily. Let $P = QDQ^T$ be the orthogonal diagonalization of $P$, and define $P^{1/2} := QD^{1/2}Q^T$, where $D^{1/2}$ denotes the diagonal matrix whose elements are the square roots of the elements of $D$. For each $n \in \N$, by the Cauchy-Schwarz inequality:
\begin{align} \label{Eqn: Prop, Change of Coordinates, 1}
    \mathbb{E}\big[\textbf{Z}[n]^T P \textbf{Z}[n] \big] &\leq \mathbb{E}\big[ \Vert P \Vert_2 \Vert \textbf{Z}[n] \Vert_2^2 \big] \\ \nonumber
    &= \Vert P \Vert_2 \cdot \mathbb{E}[\textbf{Z}[n]^T \textbf{Z}[n]]
\end{align}
and:
\begin{align} \nonumber
    \mathbb{E}\big[\textbf{Z}[n]^T \textbf{Z}[n] \big] &= \mathbb{E}\big[ (P^{1/2}\textbf{Z}[n])^T P^{-1} (P^{1/2}\textbf{Z}[n]) \big] \\ \nonumber
    &\leq \mathbb{E}\big[\Vert P^{-1} \Vert_2 \cdot \Vert P^{1/2} \textbf{Z}[n] \Vert_2^2 \big] \\ \nonumber
    &= \Vert P^{-1} \Vert_2 \cdot \mathbb{E}\big[ \Vert P^{1/2}\textbf{Z}[n] \Vert_2^2 \big]  \\ \label{Eqn: Prop, Change of Coordinates, 2}
    &= \Vert P^{-1} \Vert_2 \cdot \mathbb{E}[\textbf{Z}[n]^T P \textbf{Z}[n]].
\end{align}
Thus, since $\Vert P \Vert_2 < \infty$ and $\Vert P^{-1} \Vert_2 < \infty$, \eqref{Eqn: Prop, Change of Coordinates, 1} and \eqref{Eqn: Prop, Change of Coordinates, 2} imply that $\mathbb{E}\big[\textbf{Z}[n]^T \textbf{Z}[n] \big] < \infty$ if and only if $\mathbb{E}\big[\textbf{Z}[n]^T P \textbf{Z}[n] \big] < \infty$, and $\mathbb{E}\big[\textbf{Z}[n]^T \textbf{Z}[n] \big] = 0$ if and only if $\mathbb{E}\big[\textbf{Z}[n]^T P \textbf{Z}[n] \big] < \infty$. Similarly, $\limsup_{n \ra \infty} \mathbb{E}\big[\textbf{Z}[n]^T P \textbf{Z}[n] \big] = 0$ if and only if $\limsup_{n \ra \infty} \mathbb{E}\big[\textbf{Z}[n]^T P \textbf{Z}[n] \big] = 0$.

\subsection{Proof of Proposition \ref{Prop: Riccati Recursion, Properties}} \label{AppSubsec: Prop: Riccati Recursion, Properties}

First, we show via backward induction that each $W[n]$ is diagonal. By assumption $W[N] = P$ and $P$ is diagonal, so the assertion holds at time $N$. Suppose $W[n+1]$ is diagonal and symmetric positive definite for some $n = 0, 1, \cdots, N-1$. By \eqref{Eqn: Thm, Riccati Recursion for W, Iterative Recursion}, and the fact that $A$ is diagonal, it suffices to show that $\mathbb{E}[\textbf{M}[n]]$ is diagonal. Let $\textbf{M}_{ij}[n]$ denote the $(i, j)$-th element of $\textbf{M}[n]$ for each $i, j \in \{1, \cdots, d\}$, let $\textbf{b}_i[n]$ denote the $i$-th element of $\textbf{B}[n]$  for each $i \in \{1, \cdots, d\}$, and let $W_{ij}[n+1]$ denote the $(i, j)$-th element of $W[n+1]$ for each $i, j \in \{1, \cdots, d\}$. By hypothesis, $W[n+1]$ is diagonal, so \eqref{Eqn: Thm, Riccati Recursion for W, M[n] Definition} becomes:
\begin{align} \label{Eqn: M[i, j] Definition}
    \textbf{M}_{ij}[n] := \frac{W_{ij}[n+1]\textbf{b}_i[n] \textbf{b}_j[n]}{\sum_{k=1}^d \textbf{b}_k[n]^2 W_{kk}[n+1]}.
\end{align}
We aim to show that $\textbf{M}_{ij}[n] = 0$ whenever $i \ne j$. To establish this, fix $i, j \in \{1, \cdots, d\}, i \ne j$ arbitrarily, and
let $\textbf{b}_{-i}[n]$ denote $(\textbf{b}_1, \cdots, \textbf{b}_{i-1}, \textbf{b}_{i+1}, \cdots, \textbf{b}_d)$. Then:
\begin{align} \nonumber
    \mathbb{E}\big[ \textbf{M}_{ij}[n] \big] &= \mathbb{E}\Bigg[ \frac{W_{ij}[n+1]\textbf{b}_i[n] \textbf{b}_j[n]}{\sum_{k=1}^d \textbf{b}_k[n]^2 W_{kk}[n+1]} \Bigg] \\ \nonumber
    &= \mathbb{E}\Bigg[ \mathbb{E}\Bigg[ \frac{W_{ij}[n+1]\textbf{b}_i[n] \textbf{b}_j[n]}{\sum_{k=1}^d \textbf{b}_k[n]^2 W_{kk}[n+1]} \Bigg|\textbf{b}_{-i}[n] \Bigg] \Bigg] \\
    \label{Eqn: E[M], Equals 0}
    &= 0,
\end{align}
where 
\eqref{Eqn: E[M], Equals 0} follows because, conditioned on any fixed tuple $\textbf{b}_{-i}[n]$, the distribution of $\textbf{b}_i[n]$ is symmetric around $\textbf{b}_i[n] = 0$, while:
\begin{align}
    \frac{W_{ij}[n+1]\textbf{b}_i[n] \textbf{b}_j[n]}{\sum_{k=1}^d \textbf{b}_k[n]^2 W_{kk}[n+1]}
\end{align}
is an odd function of $\textbf{b}_i[n]$.

Next, we show via backward induction that each $W[n]$ is symmetric positive definite. First:
\begin{align*} 
    &W[n] \\
    = \hspace{0.5mm} &A^T \big( W[n+1] - \mathbb{E}[\textbf{M}[n]] W[n+1] \big) A \\
    = \hspace{0.5mm} &A^T W[n+1]^{1/2} \mathbb{E}\left[ I - \frac{W[n+1]^{1/2} \textbf{B}[n] \textbf{B}[n]^T W[n+1]^{1/2}}{\textbf{B}[n]^T W[n+1] \textbf{B}[n]} \right] \\ 
    &\hspace{5mm} W[n+1]^{1/2} A \\
    = \hspace{0.5mm} &A^T W[n+1]^{1/2} \mathbb{E}\left[ I - \overline{\textbf{B}}[n] \overline{\textbf{B}}[n]^T \right]  W[n+1]^{1/2} A,
\end{align*}
where we have defined the unit vector $\overline{\textbf{B}} \in \mathcal{S}_d$ by:
\begin{align}
    \overline{\textbf{B}}[n] := \frac{W[n+1]^{1/2} \textbf{B}[n]}{\Vert W[n+1]^{1/2} \textbf{B}[n] \Vert_2},
\end{align}
By assumption, $A$ is non-singular and $W[n+1]$ is diagonal and symmetric positive definite, it suffices to show that \mbox{$\mathbb{E}\big[I - \overline{\textbf{B}}[n]  \overline{\textbf{B}}[n]^T \big]$} is symmetric positive definite. To this end, observe that for any two unit vectors $\textbf{v}, \overline{\textbf{B}}[n] \in \mathcal{S}_d$:
\begin{align}
    &\textbf{v}^T (I - \overline{\textbf{B}}[n] \overline{\textbf{B}}[n]^T) \textbf{v} = 1 - (\overline{\textbf{B}}[n]^T \textbf{v})^2 \\
    \geq \hspace{0.5mm} &1 - \big(\Vert \overline{\textbf{B}}[n] \Vert^2 \cdot \Vert \textbf{v} \Vert^2 \big) = 0,
\end{align}
with equality if and only if $\overline{\textbf{B}}[n] = \pm \textbf{v}$, i.e. if and only if $\textbf{B}[n]$ and the nonzero vector $W[n+1]^{-1/2} v$ were parallel. Since this occurs with probability zero when $\textbf{B}[n]$ is drawn uniformly from $\mathcal{S}_d$, we conclude that:
\begin{align}
    \textbf{v}^T \mathbb{E}\big[I - \overline{\textbf{B}}[n] \overline{\textbf{B}}[n]^T \big] \textbf{v} &= \mathbb{E}\big[\textbf{v}^T \big( I - \overline{\textbf{B}}[n] \overline{\textbf{B}}[n]^T \big) \textbf{v} \big] > 0,
\end{align}
so $\mathbb{E}\big[I - \overline{\textbf{B}}[n]  \overline{\textbf{B}}[n]^T \big]$ is indeed symmetric positive definite. This concludes the proof.

\subsection{Proof of Lemma \ref{Lemma: Riccati Recursion}} \label{AppSubsec: Thm: Riccati Recursion}

The equation \eqref{Eqn: Thm, Riccati Recursion for W, at time N} establishes \eqref{Eqn: Thm, Riccati Recursion for W, Main Statement} at time $N$. Suppose \eqref{Eqn: Thm, Riccati Recursion for W, Main Statement} holds at times $n+1, \cdots, N$ for some \mbox{$n = 0, 1, \cdots, N-1$}. Then, by invoking \eqref{Eqn: Thm, Riccati Recursion for W, Iterative Recursion} at time $n+1$, we have:
\begin{align} \nonumber
    &\mathbb{E}\left[ \textbf{X}[n+1]^T W[n+1] \textbf{X}[n+1] \right] \\ \label{Eqn: Proof, Riccati Recursion, Plug in dynamics}
    = \hspace{0.5mm} &\mathbb{E}\big[ (\textbf{X}[n]^T A^T + \textbf{B}[n]^T u[n] ) W[n+1] (A \textbf{X}[n] + \textbf{B}[n] u[n] ) \big] \\ \label{Eqn: Proof, Riccati Recursion, Quadratic Cost}
    = \hspace{0.5mm} & \mathbb{E}\big[\textbf{B}[n]^T W[n+1] \textbf{B}[n] \cdot u[n]^2 + 2 \textbf{B}[n]^T W[n+1] A \textbf{X}[n] \cdot u[n] \\ \nonumber
    &\hspace{5mm} + \textbf{X}[n]^T A^T W[n+1] A \textbf{X}[n] \big], 
\end{align}
where \eqref{Eqn: Proof, Riccati Recursion, Plug in dynamics} results from substituting in the system dynamics \eqref{Eqn: System Dynamics}. Since $W[n+1]$ is symmetric positive definite, and $\Vert \textbf{B}[n] \Vert_2 = 1$, \eqref{Eqn: Proof, Riccati Recursion, Quadratic Cost} is strictly convex and quadratic in $u[n]$. Thus, by setting the derivative of \eqref{Eqn: Proof, Riccati Recursion, Quadratic Cost} with respect to $u[n]$ equal to 0, we find that its unique minimizer is given by:
\begin{align} \label{Eqn: Proof, Riccati Theorem, Minimizing u[n]}
    u^\star[n] := - \frac{\textbf{B}[n]^T W[n+1] A \textbf{X}[n]}{\textbf{B}[n]^T W[n+1] \textbf{B}[n]}
\end{align}
The corresponding minimum value of \eqref{Eqn: Proof, Riccati Recursion, Quadratic Cost} can now be obtained by substituting \mbox{$u[n] = u^\star[n]$} and the definition of $\textbf{M}[n]$ from \eqref{Eqn: Thm, Riccati Recursion for W, M[n] Definition}, and applying the law of iterated expectations:
\begin{align} 
    &\mathbb{E}\big[\textbf{X}[n]^T A^T W[n+1] \big( I - \textbf{M}[n] \big) A \textbf{X}[n] \big] \\ \nonumber
    = \hspace{0.5mm} &\mathbb{E}\big[\mathbb{E}\big[\textbf{X}[n]^T A^T W[n+1] \big( I - \textbf{M}[n] \big) A \textbf{X}[n] \mid \textbf{B}_0^{n-1} \big] \big] \\ \nonumber
    = \hspace{0.5mm} &\mathbb{E}\big[ \textbf{X}[n]^T A^T W[n+1] \big( I - \mathbb{E}\big[ \textbf{M}[n] \mid \textbf{B}_0^{n-1} \big] \big) A \textbf{X}[n]  \big] \\ \nonumber
    = \hspace{0.5mm} &\mathbb{E}\big[ \textbf{X}[n]^T A^T W[n+1] \big( I - \mathbb{E}\big[ \textbf{M}[n] \big] \big) A \textbf{X}[n] \big] \\ \label{Eqn: Proof, Riccati Theorem, Moving the Expectation, Last Expression}
    = \hspace{0.5mm} &\mathbb{E}\big[ \textbf{X}[n]^T W[n] \textbf{X}[n] \big].
\end{align}
We conclude that:
\begin{equation}
    \mathbb{E}\big[ \textbf{X}[n+1]^T W[n+1] \textbf{X}[n+1] \big] \geq \mathbb{E}\big[ \textbf{X}[n]^T W[n] \textbf{X}[n] \big]
\end{equation}
with equality if and only if the optimal control:
\begin{align}
    u[n] = u^\star[n] := - \frac{\textbf{B}[n]^T W[n+1] A \textbf{X}[n]}{\textbf{B}[n]^T W[n+1] \textbf{B}[n]}
\end{align}
is applied. This completes the recursion step and, as a result, establishes \eqref{Eqn: Thm, Riccati Recursion for W, Main Statement}.

\subsection{Proof of Lemma \ref{Lemma: Analysis, Existence of w(N, 1), w(N, 2)}:} \label{AppSubsec: Lemma: Analysis, Existence of w(N, 1), w(N, 2)}

Our proof of Lemma \ref{Lemma: Analysis, Existence of w(N, 1), w(N, 2)} relies on the following two helper lemmas:
\begin{itemize}
    \item \textbf{Lemma \ref{AppLemma: Single Parameter Real Analysis Result}}: A lemma that shows a weaker version of the desired result, with one degree of freedom used to achieve a single target equality.
    \item \textbf{Lemma \ref{Lemma: Continuous Inverse Function}}: A utility result that helps us generalize results of the form of Lemma \ref{AppLemma: Single Parameter Real Analysis Result} to greater degrees of freedom. %
\end{itemize}

Then, we will prove Lemma \ref{Lemma: Analysis, Existence of w(N, 1), w(N, 2)} itself by induction over the number of free variables ($p_i$'s) and the number of target equalities, repeatedly applying Lemma \ref{Lemma: Continuous Inverse Function} to increase the degrees of freedom from $1$ up to $d$.

\begin{lemma} \label{AppLemma: Single Parameter Real Analysis Result}
Let $\mathbf{B}$ be a vector drawn uniformly from the $d$-dimensional unit sphere. Let $\alpha_2, \cdots, \alpha_n$ be arbitrary positive reals. Define $f: \R \times \R \rightarrow \R$ by:
\begin{align} \label{AppEqn: F(p) Definition}
    f(p) := \mathbb{E}\left[ \frac{\textbf{b}_1^2}{\textbf{b}_1^2 + \sum_{k=2}^d \alpha_k \textbf{b}_k^2 \cdot p} \right].
\end{align}
We claim that, for each $v^\star \in \left(0, 1 \right)$,
there exists some $p > 0$ such that:
\begin{align}
    f(p) = v^\star.
\end{align}
Furthermore, we claim that $f(p)$ is continuous and strictly monotonically decreasing in both $p$ and each of the $\alpha_k$.
\end{lemma}
\begin{proof} (Sketch)
We have that $f(0) = 0$. As $p$ tends towards infinity, we claim that $f(p)$ approaches $1$. This is true because, for almost all realizations of $\mathbf{B}$, the quantity inside the expectation approaches $1$ as $p \to \infty$. Since the quantity inside the expectation is also always bounded, the dominated convergence theorem implies that $f(p)$ approaches $1$. Furthermore, since $f(p)$ is the expectation of a continuous function of $\mathbf{B}$, the $\{\alpha_i\}$,  and $p$, with $\mathbf{B}$ drawn from a continuous distribution, $f(p)$ is itself continuous in the $\{ \alpha_i \}$ and in $p$. Thus, by the intermediate value theorem, there exists a $p$ such that $f(p) = v^*$ for any $v^* \in (0, 1)$. Furthermore, since the quantity inside the expectation decreases monotonically with the $\{ \alpha_i \}$ and $p$ for all realizations of $\mathbf{B}$, we see that the expectation itself decreases monotonically with the $\{ \alpha_i \}$ and with $p$. A slightly more careful analysis, noting that we can lower-bound the magnitude of the derivative of the quantity inside the expectation with constant probability, shows that $f(p)$ in fact decreases strictly monotonically with $p$ and the $\{ \alpha_i \}$, establishing the final claim in our lemma.

\end{proof}

We will next establish another helper lemma.

\begin{lemma} \label{Lemma: Continuous Inverse Function}
Consider a continuous, strictly monotonically increasing function of two arguments $f: (\mathbb{R}^+, \mathbb{R}^+) \to \mathbb{R}^+$. For any $v \in (0, 1)$ and $a \in \mathbb{R}^+$, let there exist a function $g_v(a)$ such that $f(a, g_v(a)) = v$. Then, $g_v(a)$ is a continuous function decreasing strictly monotonically with $a$.
\end{lemma}
\begin{proof}
We will first show the continuity of $g$ around any $a$. Assume, for the sake of contradiction, that $g$ is not continuous. Then there must exist $a, \varepsilon > 0$, such that for every $\delta > 0$, there exists $a'$ where $|a - a'| < \delta$ such that $|g_v(a') - g_v(a)| > \varepsilon$. We will work with these $a, \varepsilon$ throughout this proof.

By the continuity of $f$, for any $\delta, \delta' > 0$, there exists $a' > 0$ where $|a' - a| < \delta$ and $|f(a', g_v(a)) - f(a, g_v(a))| < \delta'$ such that $|g_v(a') - g_v(a)| > \varepsilon$. Let this $a' = h(\delta, \delta')$ for some function $h(\cdot, \cdot)$.

Next, temporarily consider any candidate $a' \in [a/2, 3a/2]$. Let $\delta''_{a'} = \max(f(a', b + \varepsilon) - f(a', b), f(a', b) - f(a', b-\varepsilon))$. Since $f$ is strictly monotonically increasing in the second argument, $|f(a', b') - f(a', g_v(a))| > \delta''_{a'} / 2 > 0$ for all $b'$ such that $|b' - g_v(a)| > \varepsilon$. Let $\delta'' = \inf_{a' \in [a/2, 3a/2]} \delta''_{a'}/2$. Since the interval over which we are taking the infimum is closed and bounded, $\delta'' > 0$ and has the property that $|f(a', b') - f(a', g_v(a))| > \delta''$ for all $a', b'$ where $|a' - a| \le a/2$ and $|b' - g_v(a)| > \varepsilon$.

With $\delta''$ in hand, choose $\delta = a/2$ and $\delta' = \delta''/2$, and select the corresponding value $a' = h(\delta, \delta') = h(a/2, \delta''/2)$. Then we know by definition that $|g_v(a') - g_v(a)| > \varepsilon$. Yet, applying the triangle inequality,
\begin{align}
    0 &= |f(a', g_v(a')) - f(a, g_v(a))| \\
    &\ge |f(a', g_v(a')) - f(a', g_v(a))| - |f(a', g_v(a)) - f(a, g_v(a))| \\
    &\ge \delta'' - \delta' = \delta''/2 > 0,
\end{align}
a contradiction. Thus, $g$ must be continuous in $a$.

Next we will show monotonicity. For the sake of contradiction, assume that there exist some $a, \Delta$ with $\Delta > 0$ such that $g_v(a) < g_v(a + \Delta)$. Then, by definition and by the monotonicity of $f$, we find that
\[
    v = f(a, g_v(a)) < f(a, g_v(a + \Delta)) < f(a + \Delta, g_v(a+\Delta)) = v,
\]
a contradiction. Thus, we see that $g_v(a)$ must be strictly monotonically decreasing in $a$.

\end{proof}

We can now proceed to the main result.

\begin{proof} \textbf{(Proof of Lemma \ref{Lemma: Analysis, Existence of w(N, 1), w(N, 2)})}

To draw the random vector $\mathbf{B}[n]$, we can first draw a random vector $\mathbf{Z}[n]$, where each component $\mathbf{z}_i[n]$ is drawn i.i.d from a standard normal distribution, and then normalize it to unit magnitude. Observe that
\begin{align}
    \frac{\mathbf{z}_i[n]^2 \cdot p_i}{\sum_{k=1}^d \mathbf{z}_k[n]^2 \cdot p_k} = \frac{\mathbf{b}_i[n]^2 \cdot p_i}{\sum_{k=1}^d \mathbf{b}_k[n]^2 \cdot p_k}.
\end{align}
Therefore, it suffices to show that there exist positive reals $p_1, \ldots, p_d$ such that
\[
    \mathbb{E}\left[\frac{\mathbf{z}_i[n]^2 \cdot p_i}{\sum_{k=1}^d \mathbf{z}_k[n]^2 \cdot p_k}\right] = v_i
\]
for all $1 \le i \le d$. For notational convenience, we will write $\mathbf{z}_i$ in place of $\mathbf{z}_i[n]$, omitting the $[n]$, from now on.

We will now prove the following claim: for any positive coefficients $\alpha_1, \cdots, \alpha_m$, there exist positive coefficients $\beta_{m+1}, \cdots, \beta_{d}$ such that
\[
    \mathbb{E}\left[ \frac{ \beta_i \mathbf{z}_i^2 } { \sum_{k=1}^m \alpha_k \mathbf{z}_k^2 + \sum_{k={m+1}}^{d} \beta_k \mathbf{z}_k^2} \right] = v_i
\]
for $m + 1 \le i \le d$. Furthermore, we will show that the $\beta_i$ are continuous, monotonically increasing functions of the $\alpha_i$.

We will prove this by backwards induction over $m$. At the base-case of $m = d - 1$, we want to choose $\beta_d$ such that
\begin{align} \label{AppEqn: BaseCaseGoal}
    \mathbb{E}\left[\frac{ \mathbf{z}_d^2 } { \beta_d^{-1} \sum_{i=1}^{d-1} \alpha_i \mathbf{z}_i^2 + \mathbf{z}_d^2}\right] = v_d.
\end{align}
Let $\alpha_{min}$ and $\alpha_{max}$ be the smallest and largest values in $\{ \alpha_1, \cdots, \alpha_{d-1} \}$.

Observe that 
\begin{align*}
    \mathbb{E}\left[\frac{\mathbf{z}_d^2}{\beta_d^{-1}\alpha_{max} \sum_{i=1}^{d-1} \mathbf{z}_i^2 + \mathbf{z}_d^2}\right] &\le \mathbb{E}\left[\frac{\mathbf{z}_d^2}{\beta_d^{-1} \sum_{i=1}^{d-1} \alpha_i \mathbf{z}_i^2 + \mathbf{z}_d^2}\right] \\
    &\le \mathbb{E}\left[\frac{\mathbf{z}_d^2}{\beta_d^{-1}\alpha_{min} \sum_{i=1}^{d-1} \mathbf{z}_i^2 + \mathbf{z}_d^2}\right].
\end{align*}
Both sides of the above inequality can be driven to any value in the interval $(0, 1)$, from Lemma \ref{AppLemma: Single Parameter Real Analysis Result}. From Lemma \ref{AppLemma: Single Parameter Real Analysis Result}, we have that the left-hand-side of \eqref{AppEqn: BaseCaseGoal} is strictly monotonically decreasing with each of the $\alpha_i$ and in $\beta_d^{-1}$. Moreover, viewing $\beta_d^{-1}$ as the parameter to be varied in order to achieve the equality \eqref{AppEqn: BaseCaseGoal}, an extension of Lemma \ref{Lemma: Continuous Inverse Function} indicates that $\beta_d$ is a continuous, strictly monotonically increasing function of each of the $\alpha_i$. %

Now assume that our result holds for $m = k + 1$ for some integer $0 \le k \le d - 2$. We will show that it holds for $m = k$ as well. Let $p > 0$ be a parameter to be chosen later. Let $\alpha_{k+1} = p$, for the purposes of choosing $\{ \beta_{k+2}, \cdots, \beta_d \} $ using our inductive hypothesis, such that
\begin{align} \label{AppEqn: Inductive Step Alpha Beta}
    \mathbb{E}\left[\frac{ \beta_i \mathbf{z}_i^2 } { \sum_{j=1}^{k+1} \alpha_j \mathbf{z}_j^2 + \sum_{j={k+2}}^{d} \beta_j \mathbf{z}_j^2}\right] &= v_i
\end{align}
for $k + 2 \le i \le d$. By our inductive hypothesis, these $\beta_i$ are continuous, monotonically increasing functions of $\alpha_1, \cdots, \alpha_k$ and of $p$.

Let
\[
    f_k(p) = \mathbb{E}\left[\frac{p \mathbf{z}_{k+1}^2}{\sum_{i=1}^{k} \alpha_i \mathbf{z}_i^2 + p \mathbf{z}_{k+1}^2 + \sum_{i={k+2}}^{d} \beta_i \mathbf{z}_i^2}\right].
\]

When $p \to \infty$, we see that, since by the inductive hypthesis the $\beta_i$s are all positive,
\begin{align*}
    &\mathbb{E}\left[\frac{\sum_{i=1}^{k} \alpha_i \mathbf{z}_i^2}{\sum_{i=1}^{k} \alpha_i \mathbf{z}_i^2 + p \mathbf{z}_{k+1}^2 + \sum_{i={k+2}}^{d} \beta_i \mathbf{z}_i^2}\right] \\
    \le &\mathbb{E}\left[\frac{\sum_{i=1}^{k} \alpha_i \mathbf{z}_i^2}{\sum_{i=1}^{k} \alpha_i \mathbf{z}_i^2 + p \mathbf{z}_{k+1}^2}\right] \\
    =& 1 - \mathbb{E}\left[\frac{\mathbf{z}_{k+1}^2}{p^{-1}\sum_{i=1}^{k} \alpha_i \mathbf{z}_i^2 + \mathbf{z}_{k+1}^2 }\right] \to 0,
\end{align*}
applying Lemma \ref{AppLemma: Single Parameter Real Analysis Result} again in the last step (treating $p^{-1}$ as the parameter to be varied), 
Observe that
\[
    \mathbb{E}\left[\frac{\sum_{i=1}^{k} \alpha_i \mathbf{z}_i^2 + p \mathbf{z}_{k+1}^2}{\sum_{i=1}^{k} \alpha_i \mathbf{z}_i^2 + p \mathbf{z}_{k+1}^2 + \sum_{i={k+2}}^{d} \beta_i \mathbf{z}_i^2}\right] = 1 - \sum_{i = k + 2}^d v_i,
\]
by construction of the $\beta_i$. Thus, as $p \to \infty$,
\begin{align*}
    f_k(p) &= \mathbb{E}\left[\frac{\sum_{i=1}^{k} \alpha_i \mathbf{z}_i^2 + p \mathbf{z}_{k+1}^2}{\sum_{i=1}^{k} \alpha_i \mathbf{z}_i^2 + p \mathbf{z}_{k+1}^2 + \sum_{i={k+2}}^{d} \beta_i \mathbf{z}_i^2}\right] \\
    &- \mathbb{E}\left[\frac{\sum_{i=1}^{k} \alpha_i \mathbf{z}_i^2}{\sum_{i=1}^{k} \alpha_i \mathbf{z}_i^2 + p \mathbf{z}_{k+1}^2 + \sum_{i={k+2}}^{d} \beta_i \mathbf{z}_i^2}\right] \\
    \to& \left(1 - \sum_{i = k + 2}^d v_i\right) - 0 \\
    =& 1 - \sum_{i = k + 2}^d v_i.
\end{align*}
Similarly, when $p \to 0$, we have that
\[
    f_k(p) \le \mathbb{E}\left[\frac{ p \mathbf{z}_{k+1}^2 } { \sum_{i=1}^{k} \alpha_i \mathbf{z}_i^2}\right] \to 0.
\]
So by varying $p \in (0, \infty)$, we can move $f_k(p)$ arbitrarily close to both $0$ and $1 - \sum_{i=k+2}^d v_i$. Since the expected value of a continuous function of $p$ is itself continuous, we see that $f_k(\cdot)$ is a continuous function, as it depends on $p$ and on $\beta_i$ with the $\beta_i$ continuous functions of $p$ because of the inductive hypothesis. Since $0 < v_{k+1} < 1 - \sum_{i=k+2}^d v_i$, the Intermediate Value Theorem states that there exists some $p \in (0, 1)$ such that $f_k(p) = v_{k+1}$. We will choose this value of $p$ as our candidate for $\beta_{k+1}$.

Recall that for the purposes of choosing the $\beta_i$, we let $\alpha_{k+1} = p$. Since we are choosing $\beta_{k+1} = p$, we can substitute $\beta_{k+1}$ for $\alpha_{k+1} = p$ in \eqref{AppEqn: Inductive Step Alpha Beta} to obtain
\begin{align}
    &\mathbb{E}\left[\frac{ \beta_i \mathbf{z}_i^2 } { \sum_{i=1}^{k} \alpha_i \mathbf{z}_i^2 + \sum_{i={k+1}}^{d} \beta_i \mathbf{z}_i^2}\right] \\
    =& \mathbb{E}\left[\frac{ \beta_i \mathbf{z}_i^2 } { \sum_{i=1}^{k} \alpha_i \mathbf{z}_i^2 + p \mathbf{z}_{k+1} + \sum_{i={k+2}}^{d} \beta_i \mathbf{z}_i^2}\right] = v_i
\end{align}
for $k + 2 \le i \le d$ by construction of $\beta_{k+2}, \cdots, \beta_d$ using the inductive hypothesis.  Furthermore, we have by our choice of $p$ that
\[
    \mathbb{E}\left[\frac{ \beta_i \mathbf{z}_{k+1}^2 } { \sum_{i=1}^{k} \alpha_i \mathbf{z}_i^2 + \sum_{i={k+1}}^{d} \beta_i \mathbf{z}_i^2}\right] = f_k(p) = v_{k+1}.
\]
By our inductive hypothesis, we have that the $\beta_{k+2}, \cdots, \beta_d$ are continuous, monotonically increasing functions of the $\alpha_i$ and of $p$. 
So to complete the inductive step, it suffices to show that $p (=\beta_{k+1})$ is also a continuous, monotonically increasing function of each of the $\alpha_i$.  Observe that
\begin{align}
    f_k(p) =& 1 - \sum_{i = k + 2}^d v_i \\
    &- \mathbb{E}\left[\frac{\sum_{i=1}^{k} \alpha_i \mathbf{z}_i^2}{\sum_{i=1}^{k} \alpha_i \mathbf{z}_i^2 + p \mathbf{z}_{k+1} + \sum_{i={k+2}}^{d} \beta_i \mathbf{z}_i^2}\right].
\end{align}

Since $\{ \beta_{k+2}, \cdots, \beta_d \}$ are monotonically increasing functions of $p$, we see from a extension of Lemma \ref{AppLemma: Single Parameter Real Analysis Result} that $f_k(p)$ is a continuous, strictly monotonically increasing function of $p$ and a continuous, strictly monotonically decreasing function of each of the $\alpha_i$. Thus, applying Lemma \ref{Lemma: Continuous Inverse Function} (treating $p^{-1}$ as a parameter in place of $p$), we see that the optimal value of $p^{-1}$ is a continuous, strictly monotonically decreasing function of the $\alpha_i$. Thus, the optimal value of $p$ increases continuously and strictly monotonically with any of the $\{ \alpha_1, \cdots, \alpha_k \}$, as desired.

This concludes the inductive step. Our desired result follows from our inductive claim at $m = 0$.
\end{proof}

\subsection{Proof of Lemma \ref{Lemma: Main}} \label{AppSubsec: Lemma: Main}
We must first establish a helper result.

\begin{lemma} \label{Lemma: Linear Constraint}
\begin{align} \label{Eqn: Lemma, W, M pairs of elements, Linear Constraint}
    &\sum_{i=1}^d \mathbb{E}[\textbf{\emph{M}}_{ii}[n]] = \sum_{i=1}^d \mathbb{E} \left[ \frac{(W[n+1] \textbf{\emph{B}}[n]\textbf{\emph{B}}[n]^T)_{ii}}{\textbf{\emph{B}}[n]^T W[n+1] \textbf{\emph{B}}[n]} \right] = 1
\end{align}
\end{lemma}

\begin{proof}
By using the cyclic property of the trace, we have:
\begin{align*}
    \sum_{i=1}^d \mathbb{E} [\mathbf{M}_{ii}[n]] &= \sum_{i=1}^d \mathbb{E} \left[ \frac{(W[n+1] \textbf{\emph{B}}[n]\textbf{\emph{B}}[n]^T)_{ii}}{\textbf{\emph{B}}[n] W[n+1] \textbf{\emph{B}}[n]^T} \right] \\
     &= \mathbb{E} \left[ \sum_{i=1}^d\frac{(W[n+1] \textbf{\emph{B}}[n]\textbf{\emph{B}}[n]^T)_{ii}}{\textbf{\emph{B}}[n] W[n+1] \textbf{\emph{B}}[n]^T} \right] \\    
     &= \mathbb{E} \left[\frac{tr(W[n+1] \textbf{\emph{B}}[n]\textbf{\emph{B}}[n]^T)}{\textbf{B}[n] W[n+1] \textbf{\emph{B}}[n]^T}\right] \\
    &= \mathbb{E} \left[\frac{tr(\textbf{B}[n]^TW[n+1] \textbf{B}[n])}{\textbf{\emph{B}}[n]^T W[n+1] \textbf{\emph{B}}[n]}\right] \\
    &= 1.
\end{align*}
\end{proof}

With this result in hand, we can complete the proof of Lemma \ref{Lemma: Main}. We restate it below for convenience.
\lemmamain*

\begin{proof} (\textbf{Proof of Lemma \ref{Lemma: Main}})
To aid us in selecting an appropriate $P$, define:
\begin{align} \label{Eqn: v(N, 1) star formula}
    v_{i}^* = 1 - \frac{(d-1)\lambda_i^{-2}}{\sum_j \lambda_j^{-2}}
\end{align}
for all $1 \le i \le d$. Since $\lambda_i > 0$, we have that $v_{i}^* < 1$ for each $i$. From our assumption \eqref{Eqn: Lemma, Main, Constraint on lambdas}, we have that $v_{i} > 0$ for all $i$. From Lemma \ref{Lemma: Linear Constraint}, we have that $\sum_{i=1}^d m_{N-1,d} = 1$, where $m_{n, i}$ is defined to equal $\mathbb{E}[\textbf{M}_{ii}[n]]$ for convenience. Recall from the definition of $\mathbf{M}[n]$ in \eqref{Eqn: M[i, j] Definition} that $\mathbf{M}[N-1]$ is a function of $W[N] = P$ such that $m_{N-1,i}$ equals the left-hand-side of \eqref{Eqn: Lemma, Analysis, Existence of w, Eqn. 1} in the statement of Lemma \ref{Lemma: Analysis, Existence of w(N, 1), w(N, 2)}. Since the $m_{N-1,i}$ sum to $1$, Lemma \ref{Lemma: Analysis, Existence of w(N, 1), w(N, 2)} then implies the existence of some $\{ p_i^\star \}$ such that the corresponding $m_{N-1, i} = v_{i}^\star$ for all $i$.%
Now, set:
\begin{align} \nonumber
    P^\star := \text{diag}\{ p^*_1, p^*_2, \cdots, p_d^* \},
\end{align}
and define $\{W[n]\}_{n=0}^N$ by the Riccati recursion \eqref{Eqn: Thm, Riccati Recursion for W, at time N}, \eqref{Eqn: Thm, Riccati Recursion for W, Iterative Recursion}.

Below, we establish that $W[n] = r^{N-n} P$ by backwards induction from $n = N$. Since $W[N] = P$, this relation holds at time $n = N$. Now, suppose $W[n+1] = r^{N-(n+1)} P$ holds for some $n \in \{ 0, 1, \cdots, N-1 \}$. Then for any indices $i, j$:
\begin{align} \nonumber
    \frac{w_{n+1, i}^\star}{w_{n+1, j}^\star} = \frac{p_i^\star}{p_j^\star}.
\end{align}
By its definition in Proposition \ref{Prop: Riccati Recursion, Properties},
\begin{align}
    w_{n, i}^\star &= \lambda_i^2w_{n+1, i}^\star (1 - m_{n, i}).
\end{align}
Observe that $m_{n,i}$ depends only on the ratios $w_{n+1,i} / w_{n+1,k}$ for all $1 \le k \le d$. By our inductive assumption, these ratios equal $p_i^\star / p_k^\star$, and so are independent of $n$. Thus,
\[
    m_{n,i} = m_{N-1,i} = v_{i}^\star.
\]
Applying this to our recursion for $\{ w_{n, i} \}_{n=1}^\infty$, we find that
\begin{align} \label{Eqn: Lemma, Main, W}
    w^\star_{n,i} = \lambda_i^2w^\star_{n+1,i}(1 - m_{n, i}) = \lambda_i^2w^\star_{n+1,i}(1 - v_{i}^\star).
\end{align}

We claim that $W[n] = r \cdot W[n+1]$, where 
\begin{align} \nonumber
    r = \frac{d-1}{\sum_{i=1}^d \lambda_i^{-2}}.
\end{align}
This follows by rearranging terms in \eqref{Eqn: Lemma, Main, W}, where by the definitions of $v_{i}^\star$:
\begin{align*}
    \frac{w_{n, i}^\star}{w_{n+1, i}^\star} &= \lambda_i^2(1 - v_{i}^\star) = \frac{d-1}{\sum_j \lambda_j^{-2}}.
\end{align*}
Thus, $w_{n, i}^\star = p_i^\star \cdot r^{N-n}$ , with
 $ r = \frac{(d-1)}{\sum_i \lambda_i^{-2}}.$
This completes the proof.

\end{proof}

\subsection{Proof of Lemma \ref{Lemma: Embedded System}} \label{AppSubsec: Lemma: Embedded System}

Consider the simultaneous evolution of systems $\mathcal{S}$ and $\overline{\mathcal{S}}$. Let $T$ be a $(d-1) \times d$ matrix defined to be
\[
    T = \begin{bmatrix} I_{d-1} & \vec{0} \end{bmatrix}.
\]
Observe that $TA = A'T$. By Lemma \ref{AppSubsec: Lemma: Expanding or Projecting Dimension}(2), we can draw $\mathbf{B}'[n] = T \mathbf{B}[n] / \|T \mathbf{B}[n]\|$
since this yields a uniform distribution over the unit $(d-1)$-dimensional hypersphere. Let us couple $\mathcal{S}$ and $\overline{\mathcal{S}}$ in such a manner. For a given initial state $\vec{x}[0]$ for system $\mathcal{S}$, let the initial state $\vec{x}'[0]$ for $\overline{\mathcal{S}}$ be $T \vec{x}[0]$.

Imagine, for the sake of contradiction, that there exists a control strategy that can second-moment stabilize $\mathcal{S}$. Let this strategy produce control input $u[n]$ for system $\mathcal{S}$ at time $n$. Then we will construct a control strategy for the system $\overline{\mathcal{S}}$, defined to apply the input
\[
    u'[n] = u[n] \cdot \| T \mathbf{B}[n] \|
\]
at time step $n$. With this input,
\begin{align*}
    \mathbf{X}'[n+1] &= A' \mathbf{X}'[n] + (T \mathbf{B}[n] / \| T \mathbf{B}[n] \|) (u[n] \cdot \| T \mathbf{B}[n] \|) \\
    &= A' \mathbf{X}'[n] + T \mathbf{B}[n] u[n].
\end{align*}
We claim that $\mathbf{X}'[n] = T\mathbf{X}[n]$. To prove this inductively, assume that it is true for a particular $n$. Then, at timestep $n + 1$,
\begin{align*}
    \mathbf{X}'[n+1] &= A' T \mathbf{X}[n] + T \mathbf{B}[n]u[n] \\
    &= T(A\mathbf{X}[n] + \mathbf{B[n]}u[n]) \\
    &= T\mathbf{X}[n+1].
\end{align*}
Since we assumed that at $n = 0$, $\mathbf{X}'[0] = T\mathbf{X}[0]$, we have shown this claim via induction for all $n$. Now, we claim that $u'[n]$ stabilizes $\overline{\mathcal{S}}$. To see this, observe that
\[
    \mathbb{E}[\mathbf{X}[n]^2] \ge \mathbb{E}[\mathbf{X}'[n]^2]
\]
at each time-step. So if $\mathbb{E}[\mathbf{X}[n]^2] \to 0$ in the limit, so does $\mathbb{E}[\mathbf{X}'[n]^2]$, so $\overline{\mathcal{S}}$ must be stabilizable. This is a contradiction, so no control strategy can exist to stabilize $\mathcal{S}$. So $\mathcal{S}$ is not second-moment stabilizable either.

\subsection{Statement and Proof of Lemma \ref{Lemma: Expanding or Projecting Dimension}} \label{AppSubsec: Lemma: Expanding or Projecting Dimension}

We require the following facts to prove Lemma \ref{Lemma: Expanding or Projecting Dimension}.

\begin{proposition} \label{Prop: Area of a d-dimensional sphere}
Consider the change of Cartesian to generalized spherical coordinates on $\mathcal{S}^d$, denoted $(x_1, \cdots, x_d) \ra (\varphi_1, \cdots, \varphi_{d-1})$, as given below, where we have adopted the abbreviations $c_k := \cos \varphi_k$ and $s_k := \sin \varphi_k$ for each $k = 1, \cdots, d$ for brevity:
\begin{align*}
f: \hspace{2cm} \mathcal{S}^d &\ra \big( [0, \pi) \big)^{d-2} \times [0, 2\pi) \\
(x_1, \cdots, x_d) &\mapsto (\varphi_1, \cdots, \varphi_{d-1}) \\
x_1 &= c_1 \\
x_2 &= s_1 c_2 \\
x_3 &= s_1 s_2 c_3 \\
&\vdots \\
x_{d-1} &= s_1 s_2 \cdots s_{d-2} c_{d-1} \\
x_d &= s_1 s_2 \cdots s_{d-2} s_{d-1},
\end{align*}
where $\varphi_k \in [0, \pi]$ for each $k = 1, \cdots, d-2$, and $\varphi_{d-1} \in [0, 2\pi)$.

Then the area of $\mathcal{S}_d$ is given by:
\begin{align}
    \emph{Area}(\mathcal{S}_d) &= 2\pi \cdot \prod_{k=1}^{d-2} \int_0^{\pi} (\sin \varphi_k)^{d-k-1} \hspace{0.5mm} d\varphi_k.
\end{align}
\end{proposition}

\begin{proof}
Let $\Delta_d$ denote the Jacobian of the bijective map $f$, i.e.:
\begin{align}
    \Delta_d := \frac{\partial(x_1, \cdots, x_d)}{\partial(\varphi_1, \cdots \varphi_{d-1})} 
\end{align}
\begin{align*}
\Delta_3 &= \left| \begin{matrix}
c_1 & - s_1 & 0 \\
s_1 c_2 &  c_1 c_2 & - s_1 s_2 \\
s_1 s_2 &  c_1 s_2 &  s_1 c_2
\end{matrix} \right| \\
&= [c_1 \cdot  c_1 - (- s_1) \cdot s_1 (c_2^2 + s_2^2) ] \\
&= s_1,
\end{align*}
Then, for each $d \geq 3$:
\tiny
\begin{align*}
\Delta_d &= \frac{\partial(x_1, \cdots, x_d)}{\partial(\varphi_1, \cdots \varphi_{d-1})} \\
&= \left|\begin{matrix}
c_1 & - s_1 & 0 & \cdots & 0 \\
s_1 c_2 &  c_1 c_2 & - s_1 s_2 & \cdots & 0 \\
s_1 s_2 c_3 &  c_1 s_2 c_3 &  s_1 c_2 c_3 & \cdots & 0 \\
\vdots & \vdots & \vdots & \ddots & \vdots \\
s_1 s_2 s_3 \cdots c_{d-1} &  c_1 s_2 s_3 \cdots c_{d-1} &  s_1 c_2 s_3 \cdots c_{d-1} & \cdots & - s_1 s_2 s_3 \cdots s_{d-1} \\
s_1 s_2 s_3 \cdots s_{d-1} &  c_1 s_2 s_3 \cdots s_{d-1} &  s_1 c_2 s_3 \cdots s_{d-1} & \cdots &  s_1 s_2 s_3 \cdots c_{d-1} 
\end{matrix} \right| \\
&= [c_1 \cdot  c_1 \cdot s_1^{d-2} - (- s_1) \cdot s_1 \cdot s_1^{d-2}] \\
&\hspace{1cm} \cdot \left|\begin{matrix}
c_2 & - s_2 & \cdots & 0 \\
s_2 c_3 &  c_2 c_3 & \cdots & 0 \\
\vdots & \vdots & \ddots & \vdots \\
s_2 s_3 \cdots c_{d-1} &  c_2 s_3 \cdots c_{d-1} & \cdots & - s_2 s_3 \cdots s_{d-1} \\
s_2 s_3 \cdots s_{d-1} &  c_2 s_3 \cdots s_{d-1} & \cdots &  s_2 s_3 \cdots c_{d-1} 
\end{matrix} \right| \\
&=  s_1^{d-2} \cdot \Delta_{d-1}
\end{align*}
\normalsize
By induction, we conclude that:
\begin{align} \nonumber
    \Delta_d := \frac{\partial(x_1, \cdots, x_d)}{\partial(\varphi_1, \cdots \varphi_{d-1})} = s_1^{d-2} s_2^{d-3} \cdots s_{d-3}^2 s_{d-2}
\end{align}
The area of $\mathcal{S}_d$ can now be computed as follows:
\small
\begin{align*}
&\text{Area}(\mathcal{S}_d) \\
=& \int_0^\pi \cdots \int_0^\pi \int_0^{2\pi} (s_1^{d-2} s_2^{d-3} \cdots s_{d-3}^2 s_{d-2}) \hspace{0.5mm} d\varphi_{d-1} d\varphi_{d-2} \cdots d\varphi_1 \\
=& 2\pi \cdot \int_0^\pi \sin\varphi_{d-2} \hspace{0.5mm} d\varphi_{d-2} \cdot \int_0^\pi \sin^2 \varphi_{d-1} \hspace{0.5mm} d\varphi_{d-1} \cdots \int_0^\pi (\sin\varphi_1)^{d-2} \hspace{0.5mm} d\varphi_1,
\end{align*}
\normalsize
as claimed.
\end{proof}

The following is a direct corollary of the above proposition.

\begin{corollary} \label{Cor: Corollary to Prop, Area of a d-dimensional sphere}
Let $f: \mathcal{S}^d \ra \big( [0, \pi) \big)^{d-2} \times [0, 2\pi)$ denote the bijective function mapping Cartesian coordinates of points on the $d$-dimensional hypersphere, as given in Proposition \ref{Prop: Area of a d-dimensional sphere}. Fix the set (in Cartesian coordinates):
\small
\begin{align}
    S := \{ f^{-1}(S_1, \cdots, S_{d-1}) | S_1, \cdots, S_{d-2} \subset [0, \pi), S_{d-1} \subset [0, 2\pi) \} \in \mathcal{S}_d.
\end{align}
\normalsize
Then:
\begin{align}
    \emph{Area}(S) &= \ell(S_{d-1}) \cdot \prod_{k=1}^{d-2} \int_{S_k} (\sin \varphi_k)^{d-k-1} \hspace{0.5mm} d\varphi_k,
\end{align}
where $\ell(S_{d-1})$ denotes the Lebesgue measure of $S_{d-1}$.
\end{corollary}

\begin{lemma} \label{Lemma: Expanding or Projecting Dimension} ~\
\begin{enumerate}
    \item \textbf{\emph{Expansion}}---Let $\textbf{\emph{B}} := (\textbf{\emph{b}}_1, \cdots, \textbf{\emph{b}}_d) \in \R^d$ be drawn uniformly from $\mathcal{S}_d$. Let $\mathbf{\Theta} \in \R$ be drawn, independently of $\textbf{\emph{B}}$, with the density:
    \begin{align} \label{Eqn: Distribution of Theta}
        f_{\mathbf{\Theta}}(\theta) = \frac{(\sin\theta)^{d-1}}{\int_0^\pi (\sin t)^{d-1} dt}, \hspace{1cm} \theta \in [0, \pi).
    \end{align}
    and $f_{\mathbf{\Theta}}(\theta) = 0$ elsewhere. Define $\textbf{\emph{B}}' \in \R^{d+1}$ by:
    \begin{align}
        \textbf{\emph{B}}' := (\textbf{\emph{b}}_1 \sin\mathbf{\Theta}, \cdots, \textbf{\emph{b}}_d \sin \mathbf{\Theta}, \cos \mathbf{\Theta}).
    \end{align}
    Then the distribution of $\textbf{\emph{B}}'$ is uniform over the unit hypersphere $\mathcal{S}_d$.
    
    \item \textbf{\emph{Projection}}---Let  $\textbf{\emph{B}}' := (\textbf{\emph{b}}_1', \cdots, \textbf{\emph{b}}_d', \textbf{\emph{b}}_{d+1}') \in \R^{d+1}$ be drawn uniformly from $\mathcal{S}_{d+1}$, and define:
    \begin{align}
        \textbf{\emph{B}} := \frac{(\textbf{\emph{b}}_1', \cdots, \textbf{\emph{b}}_{d}')}{\Vert (\textbf{\emph{b}}_1', \cdots, \textbf{\emph{b}}_{d}') \Vert_2}.
    \end{align}
    Then the distribution of $\textbf{\emph{B}}$ is uniform over $\mathcal{S}_{d}$.
\end{enumerate}
\end{lemma}

\begin{proof} ~\
\begin{enumerate}
    \item \textbf{Expansion:}
    
    $\hspace{1cm}$ For each $d \in \N$, and any subset $M \subset \mathcal{S}_{d+1}$ of the $(d+1)$-dimensional hypersphere  $\mathcal{S}_{d+1} \in \R^{d+1}$ with well-defined surface area, denote the surface area of $M$ by Area($M$). Let $A$ be an arbitrarily selected subset of $\mathcal{S}_{d+1}$. Since we are concerned with uniform distributions over the hyperspherical surfaces, it suffices to show that:
    \begin{align}
        \Prob(\textbf{B}' \in A) = \frac{\text{Area}(A)}{\text{Area}(\mathcal{S}_{d+1})}.
    \end{align}
    Recall that any Lebesgue measurable set in any Euclidean space can be approximated as from below using open sets. Moreover, since Euclidean spaces are second countable, each open subset in an Euclidean space can be approximated by sequences of countable unions of open rectangles. Thus, without loss of generality, suppose $A$ satisfies the following property---There exist some open sets $A_s \subset \mathcal{S}_d$ and $A_\theta \subset [0, \pi) \subset \R$ such that:
    \begin{align}
        A = \big\{ (x_1 \sin\theta, \cdots, x_d \sin\theta, \cos\theta) \big| (x_1, \cdots, x_d) \in A_s, \theta \in A_\theta \big\}
    \end{align}
    Thus, we have:
    \begin{align} \nonumber
        \Prob(\textbf{B}' \in A) &= \Prob(\textbf{B} \in A_s, \mathbf{\Theta} \in A_\theta) \\ \label{Eqn: Lemma, Expanding or Projecting Dimension, Independence of Theta}
        &= \Prob(\textbf{B} \in A_s) \cdot \Prob(\mathbf{\Theta} \in A_\theta) \\ \label{Eqn: Lemma, Expanding or Projecting Dimension, Distribution of Theta}
        &= \frac{\text{Area}(A_s)}{\text{Area}(\mathcal{S}_d)} \cdot \int_{A_\theta} \frac{(\sin\theta)^{d-1}}{\int_0^\pi (\sin t)^{d-1} dt} \hspace{0.5mm} d\theta \\ \nonumber
        &= \frac{\text{Area}(A_s) \cdot \int_{A_\theta} (\sin \theta)^{d-1} d\theta}{\text{Area}(\mathcal{S}_d) \cdot \int_0^\pi (\sin t)^{d-1} dt} \\ \label{Eqn: Lemma, Expanding or Projecting Dimension, Apply Corollary}
        &= \frac{\text{Area}(A)}{\text{Area}(\mathcal{S}_{d+1})},
    \end{align}
    where \eqref{Eqn: Lemma, Expanding or Projecting Dimension, Independence of Theta} and \eqref{Eqn: Lemma, Expanding or Projecting Dimension, Distribution of Theta} follow because $\mathbf{\Theta}$ is distributed independently of \textbf{B} and according to \eqref{Eqn: Distribution of Theta}, while \eqref{Eqn: Lemma, Expanding or Projecting Dimension, Apply Corollary} follows from Corollary \ref{Cor: Corollary to Prop, Area of a d-dimensional sphere}.
    
    \item \textbf{Projection:}
    
    $\hspace{1cm}$ Let $\textbf{B}' := (\textbf{b}_1', \cdots, \textbf{b}_d', \textbf{b}_{d+1}') \in \R^{d+1}$ be any random, variable whose distribution is uniform over the unit hypersphere $\mathcal{S}_{d+1}$. We wish to show that:
    \begin{align*}
        \textbf{B} := \frac{(\textbf{b}_1', \cdots, \textbf{b}_d')}{\Vert (\textbf{b}_1', \cdots, \textbf{b}_d') \Vert_2}
    \end{align*}
    is uniformly distributed over the unit hypersphere $\mathcal{S}_d$. This is equivalent to showing that, for any subset $A_s \subset \mathcal{S}_d$ on the unit hypersphere $\mathcal{S}_d$, we have:
    \begin{align*}
        \Prob(\textbf{B} \in A_s) = \frac{\text{Area}(A_s)}{\text{Area}(\mathcal{S}_d)}.
    \end{align*}
    To show this, fix some arbitrary subset $A_s \subset \mathcal{S}_d$ on the unit hypersphere $\mathcal{S}_d$ satisfying $(0, \cdots, 0, 1), (0, \cdots, 0, -1) \not\in A_s$. Let us define a corresponding subset $A_s' \subset \mathcal{S}_{d+1}$ on the unit hypersphere $\mathcal{S}_{d+1}$ by:
    \small
    \begin{align*}
        &A_s' \\
        := \hspace{0.5mm} &\left\{ (x_1, \cdots, x_d, x_{d+1}) \in \mathcal{S}_{d+1} \Bigg| \frac{(x_1, \cdots, x_d)}{\Vert (x_1, \cdots, x_d) \Vert_2} \in A_s \right\} \\
        = \hspace{0.5mm} &\Bigg\{ (x_1', \cdots, x_d', x_{d+1}') \in \mathcal{S}_{d+1} \Bigg| \\
        &\exists \hspace{0.5mm} \theta \in [0, \pi) \text{ s.t. } \frac{1}{\sin\theta} \cdot (x_1', \cdots, x_d') \in A_s, \cos\theta = x_{d+1}' \Bigg\}
    \end{align*}
    \normalsize
    Then $\textbf{B} \in A_s$ if and only if $\textbf{B}' \in A_s'$.
    
    $\hspace{1cm}$ Next, we will leverage our results from the expansion portion of this lemma to complete the proof. Let $\overline{\textbf{B}} := (\overline{\textbf{b}_1}, \cdots, \overline{\textbf{b}_d}) \in \R^d$ be drawn uniformly from $\mathcal{S}_d$, and let $\overline{\mathbf{\Theta}} \in \R$ be drawn, independently of $\overline{\textbf{B}}$, with the density:
    \begin{align} \nonumber
        f_{\mathbf{\Theta}}(\theta) = \frac{(\sin\theta)^{d-1}}{\int_0^\pi (\sin t)^{d-1} dt}, \hspace{1cm} \theta \in [0, \pi).
    \end{align}
    and $f_{\mathbf{\Theta}}(\theta) = 0$ elsewhere. Define $\overline{\textbf{B}'} \in \R^{d+1}$ by:
    \begin{align} \nonumber
        \overline{\textbf{B}'} := (\overline{\textbf{b}_1} \sin\overline{\mathbf{\Theta}}, \cdots, \overline{\textbf{b}_d} \sin \overline{\mathbf{\Theta}}, \cos \overline{\mathbf{\Theta}}).
    \end{align}
    The proof of the expansion part of the theorem implies that $\overline{\textbf{B}'}$ is uniformly distributed. By definition, $\textbf{B}'$ is also uniformly distributed. Putting everything together, we have:
    \begin{align*}
        \Prob(\textbf{B} \in A_s) &= \Prob(\textbf{B}' \in A_s') = \Prob(\overline{\textbf{B}'} \in A_s') \\
        &= \Prob(\overline{\textbf{B}'} \in A_s, \overline{\mathbf{\Theta}} \in [0, \pi)) \\
        &= \Prob(\overline{\textbf{B}'} \in A_s) \cdot \Prob(\overline{\mathbf{\Theta}} \in [0, \pi)) \\
        &= \frac{\text{Area}(A_s)}{\text{Area}(\mathcal{S}_d)}.
    \end{align*}
    This concludes the proof.
\end{enumerate}

\end{proof}

\subsection{Proof of Theorem \ref{Thm: Main, d-dimensional}} \label{AppSubsec: Case 2}
\begin{proposition} \label{AppProp: Riccati Recursion, Properties, Dropped Controls}
Fix any $N \in \N$ and any diagonal, symmetric positive definite matrix $P$. Let $q$ be some real number in $(0, 1]$. Let the sequence of matrices $\{W_q[n] \mid n = 0, 1, \cdots, N\}$ be given by the Riccati-like recursive formula:
\begin{align} \label{AppEqn: Thm, Riccati Recursion for W, at time N}
    &W_q[N] = P, \\ \label{AppEqn: Thm, Riccati Recursion for W, Iterative Recursion}
    &W_q[n] = A^T \big( W_q[n+1] - q\mathbb{E}[\textbf{\emph{M}}[n]] W_q[n+1] \big) A,
\end{align}
where:
\begin{align} \label{AppEqn: Thm, Riccati Recursion for W, M[n] Definition}
    \textbf{\emph{M}}[n] := \frac{W_q[n+1] \textbf{\emph{B}}[n]\textbf{\emph{B}}[n]^T}{\textbf{\emph{B}}[n]^T W_q[n+1] \textbf{\emph{B}}[n]}.
\end{align}
Then each $W_q[n]$ is diagonal and symmetric positive definite.
\end{proposition}
\begin{proof}
First, we show via backward induction that each $W_q[n]$ is diagonal. By assumption $W_q[N] = P$ and $P$ is diagonal, so the assertion holds at time $N$. Suppose $W_q[n+1]$ is diagonal and symmetric positive definite for some $n = 0, 1, \cdots, N-1$. By \eqref{AppEqn: Thm, Riccati Recursion for W, Iterative Recursion}, and the fact that $A$ is diagonal, it suffices to show that $\mathbb{E}[\textbf{M}[n]]$ is diagonal. Let $\textbf{M}_{ij}[n]$ denote the $(i, j)$-th element of $\textbf{M}[n]$ for each $i, j \in \{1, \cdots, d\}$, let $\textbf{b}_i[n]$ denote the $i$-th element of $\textbf{B}[n]$  for each $i \in \{1, \cdots, d\}$, and let $W_{ij}[n+1]$ denote the $(i, j)$-th element of $W_q[n+1]$ for each $i, j \in \{1, \cdots, d\}$. By hypothesis, $W_q[n+1]$ is diagonal, so \eqref{AppEqn: Thm, Riccati Recursion for W, M[n] Definition} becomes:
\begin{align} \label{AppEqn: M[i, j] Definition}
    \textbf{M}_{ij}[n] := \frac{W_{ij}[n+1]\textbf{b}_i[n] \textbf{b}_j[n]}{\sum_{k=1}^d \textbf{b}_k[n]^2 W_{kk}[n+1]}.
\end{align}
We aim to show that $\textbf{M}_{ij}[n] = 0$ whenever $i \ne j$. To establish this, fix $i, j \in \{1, \cdots, d\}, i \ne j$ arbitrarily, and
let $\textbf{b}_{-i}[n]$ denote $(\textbf{b}_1, \cdots, \textbf{b}_{i-1}, \textbf{b}_{i+1}, \cdots, \textbf{b}_d)$. Then:
\begin{align} \nonumber
    \mathbb{E}\big[ \textbf{M}_{ij}[n] \big] &= \mathbb{E}\Bigg[ \frac{W_{ij}[n+1]\textbf{b}_i[n] \textbf{b}_j[n]}{\sum_{k=1}^d \textbf{b}_k[n]^2 W_{kk}[n+1]} \Bigg] \\ \nonumber
    &= \mathbb{E}\Bigg[ \mathbb{E}\Bigg[ \frac{W_{ij}[n+1]\textbf{b}_i[n] \textbf{b}_j[n]}{\sum_{k=1}^d \textbf{b}_k[n]^2 W_{kk}[n+1]} \Bigg|\textbf{b}_{-i}[n] \Bigg] \Bigg] \\
    \label{AppEqn: E[M], Equals 0}
    &= 0,
\end{align}
where 
\eqref{AppEqn: E[M], Equals 0} follows because, conditioned on any fixed tuple $\textbf{b}_{-i}[n]$, the distribution of $\textbf{b}_i[n]$ is symmetric around $\textbf{b}_i[n] = 0$, while:
\begin{align}
    \frac{W_{ij}[n+1]\textbf{b}_i[n] \textbf{b}_j[n]}{\sum_{k=1}^d \textbf{b}_k[n]^2 W_{kk}[n+1]}
\end{align}
is an odd function of $\textbf{b}_i[n]$.

Next, we show via backward induction that each $W_q[n]$ is symmetric positive definite. First:
\begin{align*} 
    &W_q[n] \\
    = \hspace{0.5mm} &A^T \big( W_q[n+1] - q\mathbb{E}[\textbf{M}[n]] W_q[n+1] \big) A \\
    = \hspace{0.5mm} &A^T W_q[n+1]^{1/2} q\mathbb{E}\left[ I - \frac{W_q[n+1]^{1/2} \textbf{B}[n] \textbf{B}[n]^T W_q[n+1]^{1/2}}{\textbf{B}[n]^T W_q[n+1] \textbf{B}[n]} \right] \\ 
    &\hspace{5mm} W_q[n+1]^{1/2} A \\
    = \hspace{0.5mm} &A^T W_q[n+1]^{1/2} \mathbb{E}\left[ I - q\overline{\textbf{B}}[n] \overline{\textbf{B}}[n]^T \right]  W_q[n+1]^{1/2} A,
\end{align*}
where we have defined the unit vector $\overline{\textbf{B}} \in \mathcal{S}_d$ by:
\begin{align}
    \overline{\textbf{B}}[n] := \frac{W_q[n+1]^{1/2} \textbf{B}[n]}{\Vert W_q[n+1]^{1/2} \textbf{B}[n] \Vert_2},
\end{align}
By assumption, $A$ is non-singular and $W_q[n+1]$ is diagonal and symmetric positive definite, it suffices to show that \mbox{$\mathbb{E}\big[I - q\overline{\textbf{B}}[n]  \overline{\textbf{B}}[n]^T \big]$} is symmetric positive definite. To this end, observe that for any two unit vectors $\textbf{v}, \overline{\textbf{B}}[n] \in \mathcal{S}_d$:
\begin{align}
    &\textbf{v}^T (I - q\overline{\textbf{B}}[n] \overline{\textbf{B}}[n]^T) \textbf{v} = 1 - q(\overline{\textbf{B}}[n]^T \textbf{v})^2 \\
    \geq \hspace{0.5mm} &1 - \big(\Vert \overline{\textbf{B}}[n] \Vert^2 \cdot \Vert \textbf{v} \Vert^2 \big) = 0,
\end{align}
with equality if and only if $q = 1$ and $\overline{\textbf{B}}[n] = \pm \textbf{v}$, i.e. if and only if $\textbf{B}[n]$ and the nonzero vector $W_q[n+1]^{-1/2} v$ were parallel. Since this occurs with probability zero when $\textbf{B}[n]$ is drawn uniformly from $\mathcal{S}_d$, we conclude that:
\begin{align}
    \textbf{v}^T \mathbb{E}\big[I - q\overline{\textbf{B}}[n] \overline{\textbf{B}}[n]^T \big] \textbf{v} &= \mathbb{E}\big[\textbf{v}^T \big( I - q\overline{\textbf{B}}[n] \overline{\textbf{B}}[n]^T \big) \textbf{v} \big] > 0,
\end{align}
so $\mathbb{E}\big[I - q\overline{\textbf{B}}[n]  \overline{\textbf{B}}[n]^T \big]$ is indeed symmetric positive definite. This concludes the proof.
\end{proof}

\begin{lemma} \label{AppLemma: Riccati Recursion}
Consider the system $\mathcal{S}$ starting at the fixed, known initial state $X[0]$ at time $0$, except that with probability $1-p$, we are required to set $u[n] = 0$ (so are unable to apply any control input to the system). Fix any diagonal, symmetric positive definite matrix $P \in \R^{d \times d}$. Then, for each $N, n \in \N$, $n = 0, 1, \cdots, N$:
\begin{equation} \label{AppEqn: Thm, Riccati Recursion for W, Main Statement}
    \min_{g_0^N \in \mathcal{G}_0^N} \mathbb{E}\left[ \textbf{\emph{X}}[N]^T P \textbf{\emph{X}}[N] \right] = \min_{g_0^n \in \mathcal{G}_0^n} \mathbb{E}\left[ \textbf{\emph{X}}[n]^T W_q[n] \textbf{\emph{X}}[n] \right],
\end{equation}
where \mbox{$\{W_q[n] \mid n = 0, 1, \cdots, N\}$} is given by the recursive formulas \eqref{AppEqn: Thm, Riccati Recursion for W, at time N} and \eqref{AppEqn: Thm, Riccati Recursion for W, Iterative Recursion}, and $W_q[N] = P$. In particular:
\begin{align} \nonumber
    \min_{g_0^N \in \mathcal{G}_0^N} \mathbb{E}\left[ \textbf{\emph{X}}[N]^T P \textbf{\emph{X}}[N] \right] =  \textbf{\emph{X}}[0]^T W_q[0] \textbf{\emph{X}}[0].
\end{align}
\end{lemma}

\begin{proof}
The equation \eqref{AppEqn: Thm, Riccati Recursion for W, at time N} establishes \eqref{AppEqn: Thm, Riccati Recursion for W, Main Statement} at time $N$. Suppose \eqref{AppEqn: Thm, Riccati Recursion for W, Main Statement} holds at times $n+1, \cdots, N$ for some \mbox{$n = 0, 1, \cdots, N-1$}. Then, by invoking \eqref{AppEqn: Thm, Riccati Recursion for W, Iterative Recursion} at time $n+1$, \emph{conditioned on us being in the scenario with probability $p$ when we can apply controls}, we have:
\begin{align} \nonumber
    &\mathbb{E}\left[ \textbf{X}[n+1]^T W_q[n+1] \textbf{X}[n+1] \right] \\ \label{AppEqn: Proof, Riccati Recursion, Plug in dynamics}
    = \hspace{0.5mm} &\mathbb{E}\big[ (\textbf{X}[n]^T A^T + \textbf{B}[n]^T u[n] ) W_q[n+1] (A \textbf{X}[n] + \textbf{B}[n] u[n] ) \big] \\ \label{AppEqn: Proof, Riccati Recursion, Quadratic Cost}
    = \hspace{0.5mm} & \mathbb{E}\big[\textbf{B}[n]^T W_q[n+1] \textbf{B}[n] \cdot u[n]^2 \\ \nonumber
    &\hspace{5mm} + 2 \textbf{B}[n]^T W_q[n+1] A \textbf{X}[n] \cdot u[n] \\ \nonumber
    &\hspace{5mm} + \textbf{X}[n]^T A^T W_q[n+1] A \textbf{X}[n] \big], 
\end{align}
where \eqref{AppEqn: Proof, Riccati Recursion, Plug in dynamics} results from substituting in the system dynamics \eqref{Eqn: System Dynamics}. Since $W_q[n+1]$ is symmetric positive definite, and $\Vert \textbf{B}[n] \Vert_2 = 1$, \eqref{AppEqn: Proof, Riccati Recursion, Quadratic Cost} is strictly convex and quadratic in $u[n]$. Thus, by setting the derivative of \eqref{AppEqn: Proof, Riccati Recursion, Quadratic Cost} with respect to $u[n]$ equal to 0, we find that its unique minimizer is given by:
\begin{align} \label{AppEqn: Proof, Riccati Theorem, Minimizing u[n]}
    u^\star[n] := - \frac{\textbf{B}[n]^T W_q[n+1] A \textbf{X}[n]}{\textbf{B}[n]^T W_q[n+1] \textbf{B}[n]}
\end{align}
The corresponding minimum value of \eqref{AppEqn: Proof, Riccati Recursion, Quadratic Cost} can now be obtained by substituting \mbox{$u[n] = u^\star[n]$} and the definition of $\textbf{M}[n]$ from \eqref{AppEqn: Thm, Riccati Recursion for W, M[n] Definition}, and applying the law of iterated expectations, no longer conditioning that we are in the scenario where we can apply control:
\begin{align} 
    &p\mathbb{E}\Bigg[ \textbf{X}[n]^T A^T \\ \nonumber
    &\hspace{5mm} \Bigg( W_q[n+1] - \frac{W_q[n+1] \textbf{B}[n] \textbf{B}[n]^T W_q[n+1]}{ \textbf{B}[n]^T W_q[n+1] \textbf{B}[n]} \Bigg) A \textbf{X}[n] \Bigg] \\ %
    &\hspace{5mm} + (1-p)\mathbb{E}\left[ \mathbf{X}[n]^TA^TW_q[n+1]A\mathbf{X}[n]\right]\nonumber \\
    = \hspace{0.5mm} &\mathbb{E}\big[\textbf{X}[n]^T A^T W_q[n+1] \big( I - p\textbf{M}[n] \big) A \textbf{X}[n] \big] \\ 
    = \hspace{0.5mm} &\mathbb{E}\big[\mathbb{E}\big[\textbf{X}[n]^T A^T W_q[n+1] \big( I - p\textbf{M}[n] \big) A \textbf{X}[n] \mid \textbf{B}_0^{n-1} \big] \big] \\ \nonumber
    = \hspace{0.5mm} &\mathbb{E}\big[ \textbf{X}[n]^T A^T W_q[n+1] \big( I - p\mathbb{E}\big[ \textbf{M}[n] \mid \textbf{B}_0^{n-1} \big] \big) A \textbf{X}[n]  \big] \\ \nonumber
    = \hspace{0.5mm} &\mathbb{E}\big[ \textbf{X}[n]^T A^T W_q[n+1] \big( I - p\mathbb{E}\big[ \textbf{M}[n] \big] \big) A \textbf{X}[n] \big] \\ \label{AppEqn: Proof, Riccati Theorem, Moving the Expectation, Last Expression}
    = \hspace{0.5mm} &\mathbb{E}\big[ \textbf{X}[n]^T W_q[n] \textbf{X}[n] \big].
\end{align}
We conclude that:
\begin{equation}
    \mathbb{E}\big[ \textbf{X}[n+1]^T W_q[n+1] \textbf{X}[n+1] \big] \geq \mathbb{E}\big[ \textbf{X}[n]^T W_q[n] \textbf{X}[n] \big]
\end{equation}
with equality if and only if the optimal control:
\begin{align} \label{AppEqn: Optimal Control}
    u[n] = u^\star[n] := - \frac{\textbf{B}[n]^T W_q[n+1] A \textbf{X}[n]}{\textbf{B}[n]^T W_q[n+1] \textbf{B}[n]}
\end{align}
is applied. This completes the recursion step and, as a result, establishes \eqref{AppEqn: Thm, Riccati Recursion for W, Main Statement}.

\end{proof}

\begin{lemma} \label{AppLemma: Main}
Consider the system $\mathcal{S}$ with controls dropped at each timestep with probability $1 - q$, as in Lemma \ref{AppLemma: Riccati Recursion}, and with some fixed, known initial state $X[0]$. Then we can show that there exists some $r \in \R$, $r > 0$ and some diagonal, symmetric positive definite matrix $P \in \R^{d \times d}$:
    \begin{align} \label{AppEqn: Lemma, Main, P Definition}
        P := \text{diag}\{ p_1, p_2, \cdots, p_d \}
    \end{align}
    such that the matrices $W_q[N], \cdots, W_q[0]$ generated by the Riccati recursion \eqref{AppEqn: Thm, Riccati Recursion for W, at time N} from $W_q[N] = P$, additionally satisfy:
    \begin{align} \label{AppEqn: Lemma, Main, W_q[n] sequence is geometric}
        W_q[n] = r^{N-n} \cdot P.
    \end{align}
    for each $n = 0, 1, \cdots, N$, for all $A$ where
    \begin{align} \label{AppEqn: Lemma, Main, Constraint on lambdas}
        1 - \frac{(d-q)\lambda_i^{-2}}{\sum_{j=1}^d \lambda_j^{-2}} > 0
    \end{align}
    for all $1 \le i \le d$.
    
Moreover, $r$ is given by:
\begin{align} \label{AppEqn: Lemma, Main, Formula for r}
    r = \frac{d-q}{\sum_{i=1}^d \lambda_i^{-2}}
\end{align}
\end{lemma}

\begin{proof} (\textbf{Proof of Lemma \ref{AppLemma: Main}})
To aid us in selecting an appropriate $P$, define:
\begin{align} \label{AppEqn: v(N, 1) star formula}
    v_i^* = 1 - \frac{(d-q)\lambda_i^{-2}}{\sum_j \lambda_j^{-2}}
\end{align}
for all $1 \le i \le d$. Since $\lambda_i > 0$, we have that $v_{i}^* < 1$ for each $i$. From our assumption \eqref{AppEqn: Lemma, Main, Constraint on lambdas}, we have that $v_{i} > 0$ for all $i$. Lemma \ref{Lemma: Analysis, Existence of w(N, 1), w(N, 2)} then implies the existence of some $\{ p_i^\star \}$ such that the corresponding $m_{N-1, i} = v_{i}^\star$ for all $i$, where $m_{n, i}$ is defined to equal $\mathbb{E}[\textbf{M}_{ii}[n]]$. Now, set:
\begin{align} \nonumber
    P^\star := \text{diag}\{ p^*_1, p^*_2, \cdots, p_d^* \},
\end{align}
and define $\{W_q[n]\}_{n=0}^N$ by the Riccati recursion \eqref{AppEqn: Thm, Riccati Recursion for W, at time N}, \eqref{AppEqn: Thm, Riccati Recursion for W, Iterative Recursion}.

Below, we establish \eqref{AppEqn: Lemma, Main, W_q[n] sequence is geometric} by backwards induction from $n = N$. Since $W_q[N] = P$, \eqref{AppEqn: Lemma, Main, W_q[n] sequence is geometric} holds at time $n = N$. Now, suppose \eqref{AppEqn: Lemma, Main, W_q[n] sequence is geometric} holds at time $n+1$, for some $n = 0, 1, \cdots, N-1$. Then for any indices $i, j$:
\begin{align} \nonumber
    \frac{w_{n+1, i}^\star}{w_{n+1, j}^\star} = \frac{p_i^\star}{p_j^\star}.
\end{align}
By its definition in Proposition \ref{Prop: Riccati Recursion, Properties},
\begin{align}
    w_{n, i}^\star &= \lambda_i^2w_{n+1, i}^\star (1 - qm_{n, i}).
\end{align}
Observe that $m_{n,i}$ depends only on the ratios $w_{n+1,i} / w_{n+1,k}$ for all $1 \le k \le d$. By our inductive assumption, these ratios equal $p_i^\star / p_k^\star$, and so are independent of $n$. Thus,
\[
    m_{n,i} = m_{N-1,i} = v^\star_i.
\]
Applying this to our recursion for $\{ w_{n, i} \}_{n=1}^\infty$, we find that
\begin{align} \label{AppEqn: Lemma, Main, W}
    w^\star_{n,i} = \lambda_i^2w^\star_{n+1,i}(1 - qm_{n, i}) = \lambda_i^2w^\star_{n+1,i}(1 - qv_{i}^\star).
\end{align}

We claim that $W_q[n] = r \cdot W_q[n+1]$, where 
\begin{align} \nonumber
    r = \frac{d-q}{\sum_{i=1}^d \lambda_i^{-2}}.
\end{align}
This follows by rearranging terms in \eqref{AppEqn: Lemma, Main, W}, where by the definitions of $v_{i}^\star$:
\begin{align*}
    \frac{w_{n, i}^\star}{w_{n+1, i}^\star} &= \lambda_i^2(1 - v_{i}^\star) = \frac{d-q}{\sum_j \lambda_j^{-2}}.
\end{align*}
Thus, $w_{n, i}^\star = p_i^\star \cdot r^{N-n}$ , with
 $ r = \frac{(d-q)}{\sum_i \lambda_i^{-2}}.$
This completes the proof.

\end{proof}

\begin{proof} \textbf{(Proof of Theorem \ref{Thm: Main, d-dimensional} in Case 2)}
Without loss of generality, we can express the state matrix $A$ of our system $\mathcal{S}$ as
\begin{align} \label{AppEqn: Sub-State Matrix}
    A := \begin{bmatrix}
    A' & O \\
    O & \Lambda
    \end{bmatrix},
\end{align}
where all the eigenvalues $\lambda_1, \cdots, \lambda_m$ of $A'$ are greater than $1$ in magnitude, while all the eigenvalues of $\Lambda = diag(\lambda_{m+1}, \ldots, \lambda_d) \in \R$ are strictly less than $1$ in magnitude.

Define the system $\overline{\mathcal{S}}$ to be:
\begin{align} \label{AppEqn: Embedded System}
    \overline{\mathcal{S}}: \hspace{5mm} \textbf{\emph{X}}'[n+1] &= A' \textbf{\emph{X}}'[n] + \textbf{\emph{B}}'[n]u[n], \\ \nonumber 
    \textbf{\emph{Y}}'[n] &= \textbf{\emph{X}}'[n].
\end{align}
To prove the second part of the lemma, observe that if $r > 1$, then by Theorem \ref{Thm: Main, d-dimensional} in Case 1, $\overline{\mathcal{S}}$ is not second-moment stabilizable. Then, by Lemma \ref{Lemma: Embedded System}, $\mathcal{S}$ is not second-moment stabilizable either. Thus, if $\mathcal{S}$ is second-moment stabilizable, $\overline{\mathcal{S}}$ must be second-moment stabilizable too, so $r \le 1$.

Now we will show that $r < 1$ implies that $\mathcal{S}$ is second-moment stabilizable. At a high-level, we will consider the embedded system $\overline{\mathcal{S}}$ within $\mathcal{S}$, which we know from Lemma \ref{AppLemma: Main} is stabilizable with some control strategy $u'[n]$. We will essentially use $u'[n]$ to stabilize the full system $\mathcal{S}$, after rescaling it appropriately to form some $u[n]$. For the $d-m$ dimensions in $\mathcal{S}$ but not in $\overline{\mathcal{S}}$, we will rely on their $|\lambda_i| < 1$ to let the component of the state in that direction decay over time. However, rescaling $u'[n]$ to form $u[n]$ means that $u[n]$ could be unbounded, since the rescaling factor could be arbitrarily large. To handle this, we will simply ``drop'' our control input (setting $u[n] = u'[n] = 0$) if the rescaling factor exceeds a threshold. To ensure that this control dropping does not prevent us from stabilizing $\overline{\mathcal{S}}$, we will choose our threshold appropriately, relying on the ``slack'' available in $\overline{\mathcal{S}}$, since $r$ is strictly less than $1$, so we can afford to increase it by some small amount without causing the subsystem to become unstabilizable.

Consider the simultaneous evolution of systems $\mathcal{S}$ and $\overline{\mathcal{S}}$. Let $T$ be an $m \times d$ matrix defined to be
\[
    T = \begin{bmatrix} I_{m} & O \end{bmatrix}.
\]
Observe that $TA = A'T$. By Lemma \ref{AppSubsec: Lemma: Expanding or Projecting Dimension}(2), we can draw $\mathbf{B}'[n] = T \mathbf{B}[n] / \|T \mathbf{B}[n]\|$
since this yields a uniform distribution over the unit $m$-dimensional hypersphere. Let us couple $\mathcal{S}$ and $\overline{\mathcal{S}}$ in such a manner. For a given initial state $\vec{x}[0]$ for system $\mathcal{S}$, let the initial state $\vec{x}'[0]$ for $\overline{\mathcal{S}}$ be $T \vec{x}[0]$.

Consider a pair of control strategies for our two systems, producing control input $u[n]$ for system $\mathcal{S}$ and $u'[n]$ for system $\overline{\mathcal{S}}$ at time $n$, related by
\begin{align} \label{AppEqn: Coupled Control Relation}
    u'[n] = u[n] \cdot \| T \mathbf{B}[n] \|
\end{align}
at time step $n$. With these inputs,
\begin{align*}
    \mathbf{X}'[n+1] &= A' \mathbf{X}'[n] + (T \mathbf{B}[n] / \| T \mathbf{B}[n] \|) (u[n] \cdot \| T \mathbf{B}[n] \|) \\
    &= A' \mathbf{X}'[n] + T \mathbf{B}[n] u[n].
\end{align*}
We claim that $\mathbf{X}'[n] = T\mathbf{X}[n]$. To prove this inductively, assume that it is true for a particular $n$. Then, at timestep $n + 1$,
\begin{align}
    \mathbf{X}'[n+1] &= A' T \mathbf{X}[n] + T \mathbf{B}[n]u[n] \\
    &= T(A\mathbf{X}[n] + \mathbf{B[n]}u[n]) \\
    &= T\mathbf{X}[n+1].
\end{align}
Since we assumed that at $n = 0$, $\mathbf{X}'[0] = T\mathbf{X}[0]$, we have shown this claim via induction for all $n$.

Let $h \in (1, \infty)$ be a parameter to be chosen later. If $1/\| T \mathbf{B}[n] \| \ge h$, we will ``drop'' the control input to $\overline{\mathcal{S}}$. Let $q' = \mathbb{P}(1/\| T \mathbf{B}[n] \| \ge h)$.

Let $q$ be a parameter in $(0, 1)$ and let
\[
    r' = \frac{m-q}{\sum_{i=1}^m \lambda_i^{-2}} > r.
\]
Since $r < 1$, there exists a choice of $q$ such that $r' < 1$. Consider this value of $q$. Observe that $q'$ is a continuous function of $h$, since the $\mathbf{B}[n]$ have a continuous distribution. As $h \to 1$, $q' \to 1$, and as $h \to \infty$, $p' \to 0$. So there exists some choice of $h$ such that $q = 1 - q'$. We will choose this value of $h$. Therefore, the probability of our control of $\overline{\mathcal{S}}$ being dropped is $1 - q$.

Since $r' < 1$ and $|\lambda_i| > 1$ for $1 \le i \le m$, we have that
\[
    v_i^* = 1 - \frac{(d-q)\lambda_i^{-2}}{\sum_{j=1}^m \lambda_j^{-2}} = 1 - \lambda_i^{-2}r' > 0
\]
for $1 \le i \le m$. Therefore, from Lemma \ref{AppLemma: Main}, since $r' < 1$, there exists a control strategy to stabilize the system $\overline{\mathcal{S}}$, even if our control input is dropped with probability $q'$.

We will choose our control input $u[n]$ at each time step in the following manner. If $1/\| T \mathbf{B}[n] \| \ge h$, apply $u[n] = 0$ (and $u'[n] = 0$ in the coupled system $\overline{\mathcal{S}}$). Otherwise, consider the $u'[n]$ applied by our known control strategy \eqref{AppEqn: Optimal Control} in the coupled system $\overline{\mathcal{S}}$. Then, for the system $\mathcal{S}$, apply control input $u[n] = u'[n] / \| T \mathbf{B}[n] \|$. 

We wish to show that this control strategy second-moment stabilizes $\mathcal{S}$. Since we drop our control with probability $q'$ at each time step, Lemma \ref{AppLemma: Main} states that the coupled system $\overline{\mathcal{S}}$ is stabilized by our control strategy. Since $u[n]$ and $u'[n]$ are related by \eqref{AppEqn: Coupled Control Relation}, we have that $\mathbf{X}'[n] = T \mathbf{X}[n]$ at every time step. 

We know that the optimal control $u'[n]$ applied is as stated in \eqref{AppEqn: Optimal Control}, where the $W_q[n]$ are determined using the Riccati recursion over the subsystem $\overline{\mathcal{S}}$, starting from the $P$ constructed in Lemma \ref{AppLemma: Main}. 

From \eqref{AppEqn: Optimal Control}, we have that
\begin{align} \label{AppEqn: Embedded control magnitudes}
    \mathbf{B}'[n] u'[n] = \mathbf{B}'[n] \frac{\textbf{B}'[n]^T W_q[n+1] A \textbf{X}'[n]}{\textbf{B}'[n]^T W_q[n+1] \textbf{B}'[n]}.
\end{align}
We claim that the magnitude of this vector is bounded by $K \| \mathbf{X}'[n] \|$ for some $K > 0$ independent of $n$. To see this, recall that $W_q[n] = r^{N-n}P$, so we can bound
\[
    \| \mathbf{B}'[n] u'[n] \| = \frac{\|\textbf{B}'[n]^T W_q[n+1] A \textbf{X}'[n]\|}{\|\textbf{B}'[n]^T W_q[n+1] \textbf{B}'[n]\|}.
\]
Observe that
\begin{align*}
    &\| \mathbf{B}'[n]^TW_q[n+1]A\mathbf{X}'[n] \| \\
    \le& \| \mathbf{B}'[n] \| \cdot \| W_q[n+1]A\mathbf{X}'[n] \| \\
    \le& \| \mathbf{B}'[n] \| \cdot \sqrt{\lambda_{max}(W_q[n+1]A)} \cdot \| \mathbf{X}'[n] \| \\
    =& \sqrt{\lambda_{max}(W_q[n+1]A)} \cdot \| \mathbf{X}'[n] \|.
\end{align*}
Thus,
\begin{align}
    \| \mathbf{B}'[n] u'[n] \| &\le \frac{\sqrt{\lambda_{max}(W_q[n+1]A)} \| \mathbf{X}'[n] \|}{\sqrt{\lambda_{min}(W_q[n+1])}} \\
    &\le \sqrt{\frac{\lambda_{max}(r^{N-n}P) \lambda_{max}(A)}{\lambda_{min}(r^{N-n}P)}} \| \mathbf{X}'[n] \| \label{AppEqn: Max Lambda Product} \\
    &= \sqrt{\frac{\lambda_{max}(P) \lambda_{max}(A)}{\lambda_{min}(P)}} \| \mathbf{X}'[n] \| \\
    &= K \| \mathbf{X}'[n] \|
\end{align}
for some constant $K$ independent of $n$. Recall that $P$ and $A$ are both diagonal matrices, justifying \eqref{AppEqn: Max Lambda Product}.

Since $\overline{\mathcal{S}}$ is stabilizable under the control input $u'[n]$, $\mathbb{E}[\| \mathbf{X}'[n] \|^2]$ is bounded in magnitude by some $L$. Thus, $\mathbb{E}[u'[n]^2]$ is bounded in magnitude by $KL$.

Thus, to show that $\mathbb{E}[\mathbf{X}[n]^2]$ does not go to infinity, it remains to show that $\mathbb{E}[\mathbf{X}_i[n]^2]$ is bounded as $n \to \infty$ for $m + 1 \le i \le d$. Consider any such $i$. We can express
\[
    x_i[n+1] = \lambda_i x_i[n] + u[n] \mathbf{b}_i.
\]
Since $u[n]$ depends only on the magnitude of $\mathbf{B}$, and on the values of its first $m$ components, it is independent of the sign of $\mathbf{b}_i$, so we may write
\begin{align*}
    \mathbb{E}[x_i^2[n+1]] =& \lambda_i^2 \mathbb{E}[x_i[n]^2] + \mathbb{E}[u[n]^2\mathbf{b}_i^2] \\
    &+ 2\lambda_i \mathbb{E}[x_i[n]u[n] \cdot |\mathbf{b}_i|] \mathbb{E}[\text{sign}(\mathbf{b}_i)] \\
    \le& \lambda_i^2 \mathbb{E}[x_i^2[n]] + (hKL)^2.
\end{align*}
Evaluating this recurrence, we find that
\begin{align*}
    \mathbb{E}[x_i[n]^2] &= \mathbb{E}[x_i[0]^2] \lambda^{2n} + (hKL)^2 \frac{\lambda_i^{2n} - 1}{\lambda_i^2 - 1}.
\end{align*}
Since $|\lambda_i| < 1$ for $m + 1 \le i \le d$, this expectation does not go to infinity in the limit. Since this is true for all $m + 1 \le i \le d$, $\mathcal{S}$ is second-moment stabilizable if $r < 1$, completing the proof.
\end{proof}

\addtolength{\textheight}{-3cm}

\end{document}